\documentclass[leqno,11pt]{article}

\usepackage[utf8]{inputenc}
\usepackage[T1]{fontenc}
\usepackage[english]{babel}

\usepackage[intlimits]{amsmath}
\usepackage{amsthm}
\usepackage{amsbsy}
\usepackage{amscd} 
\usepackage{amsfonts}
\usepackage{amsgen} 
\usepackage{amsmath}
\usepackage{amsopn} 
\usepackage{amssymb}
\usepackage{amstext}
\usepackage{amsthm} 
\usepackage{amsxtra}

\usepackage{MnSymbol}
\usepackage{mathbbol}
\usepackage{bbm}
\usepackage{mathrsfs}

\usepackage[all]{xy}

\usepackage{graphicx}
\usepackage{rotating}
\usepackage{graphics}
\usepackage{color}

\numberwithin{equation}{section}

\newtheorem{theoremcounter}{theoremcounter}[section]
\newtheorem{thmstarcounter}{thmstarcounter}

\newtheorem{corollary}[theoremcounter]{Corollary}
\newtheorem{lemma}[theoremcounter]{Lemma}

\newtheorem{proposition}[theoremcounter]{Proposition}
\newtheorem{theorem}[theoremcounter]{Theorem}
\newtheorem{thmstar}[thmstarcounter]{Theorem}

\theoremstyle{definition}

\newtheorem{definition}[theoremcounter]{Definition}

\newtheorem{example}[theoremcounter]{Example}
\newtheorem{remark}[theoremcounter]{Remark}

\newcommand{\cC}{\ensuremath{\mathcal{C}}}

\newcommand{\cP}{\ensuremath{\mathcal{P}}}

\newcommand{\cR}{\ensuremath{\mathcal{R}}}
\newcommand{\cS}{\ensuremath{\mathcal{S}}}

\newcommand{\cV}{\ensuremath{\mathcal{V}}}

\newcommand{\rC}{\ensuremath{\mathrm{C}}}

\newcommand{\rM}{\ensuremath{\mathrm{M}}}

\newcommand{\amid}{\ensuremath{\, | \,}}

\newcommand{\eqstop}{\ensuremath{\, \text{.}}}
\newcommand{\eqcomma}{\ensuremath{\, \text{,}}}

\newcommand{\NN}{\ensuremath{\mathbb{N}}}
\newcommand{\ZZ}{\ensuremath{\mathbb{Z}}}

\newcommand{\CC}{\ensuremath{\mathbb{C}}}

\newcommand{\Hom}{\ensuremath{\mathop{\mathrm{Hom}}}}
\newcommand{\id}{\ensuremath{\mathrm{id}}}
\newcommand{\ra}{\ensuremath{\rightarrow}}
\newcommand{\lra}{\ensuremath{\longrightarrow}}
\newcommand{\hra}{\ensuremath{\hookrightarrow}}
\newcommand{\thra}{\ensuremath{\twoheadrightarrow}}

\newcommand{\Cmat}[1]{\ensuremath{\mathrm{M}_{#1}(\CC)}}

\newcommand{\GL}[2]{\ensuremath{\mathrm{GL}_{#1}(#2)}}

\newcommand{\End}{\ensuremath{\mathrm{End}}}

\newcommand{\ot}{\ensuremath{\otimes}}

\newcommand{\Cstar}{\ensuremath{\text{C}^*}}

\newcommand{\bo}{\ensuremath{\mathscr{B}}}

\newcommand{\FdHilb}{\ensuremath{\mathrm{FdHilb}}}

\newcommand{\Cstarmax}{\ensuremath{\Cstar_\mathrm{max}}}

\newcommand{\lspan}{\ensuremath{\mathop{\mathrm{span}}}}

\newcommand{\freegrp}[1]{\ensuremath{\mathbb{F}}_{#1}}

\newcommand{\Corepfin}{\ensuremath{\mathrm{CoRep}_{\mathrm{fin}}}}

\renewcommand{\theta}{\vartheta}
\renewcommand{\phi}{\varphi}
\renewcommand{\epsilon}{\varepsilon}
\renewcommand{\subset}{\subseteq}
\renewcommand{\supset}{\supseteq}

\newcommand{\N}{\mathbb N}
\newcommand{\Z}{\mathbb Z}

\newcommand{\C}{\mathbb C}
\newcommand{\F}{\mathbb F}

\newcommand{\Katprim}{\langle\primarypart\rangle}
\newcommand{\Katfat}{\langle\fatcrosspart\rangle}
\newcommand{\AKat}{A_{\mathcal C}(n)}
\newcommand{\AKatprime}{A_{\mathcal C'}(n)}

\newcommand{\idpart}{|}
\newcommand{\paarpart}{\sqcap}
\newcommand{\baarpartbaustein}{\rotatebox{180}{$\sqcap$}}
\newcommand{\baarpart}{
\mathrel{\vcenter{\offinterlineskip \hbox{$\baarpartbaustein$}}}}
\newcommand{\upsubset}{\begin{rotate}{90}$\subset$\end{rotate}}
\newcommand{\downsubset}{\begin{turn}{270}$\subset$\end{turn}}
\newcommand{\singleton}{\uparrow}
\newcommand{\doublesingleton}{\singleton\otimes\singleton}

\newcommand{\vierpart}{
\mathrel{\offinterlineskip
\hskip0ex\hbox{$\sqcap$}\hskip -.4ex\hbox{$\sqcap$} \hskip -0.4ex\hbox{$\sqcap$}}}

\newcommand{\vierpartrot}{
\mathrel{\vcenter{\offinterlineskip
\hbox{$\baarpart$} \vskip -.1ex \hbox{$\shortmid$} \vskip -.1ex \hbox{$\paarpart$}}}}

\newcommand{\crosspart}{
\mathrel{\offinterlineskip
\hbox{$/$}\hskip -.95ex\hbox{$\backslash$}}}

\newcommand{\fatcrosspart}{
\mathrel{\vcenter{\offinterlineskip
\hbox{$\baarpart\baarpart$} \vskip -0.2ex \hbox{\hskip .75ex $\times$} \vskip -.2ex \hbox{$\paarpart\paarpart$}}}}

\newcommand{\primarypart}{
\mathrel{\vcenter{\offinterlineskip
\hbox{$\baarpart$} \vskip -1.3ex \hbox{\hskip1.3ex$/$\hskip-1.2ex$-$} \vskip -1.2ex \hbox{\hskip2.2ex $\paarpart$}}}}

\newcommand{\midmid}{
\mathrel{\vcenter{\offinterlineskip
\hbox{$\shortmid$} \vskip -1.6ex \hbox{$\shortmid$}}}}

\newcommand{\halflibpart}{
\mathrel{\offinterlineskip
\hbox{$\bigtimes$}\hskip -1.55ex\hbox{$\midmid$}}}

\newcommand{\liegeblkn}{\begin{turn}{270}$[$\end{turn}}

\newcommand{\liegebalken}{
\mathrel{\vcenter{\offinterlineskip
\vskip -1.3ex \hbox{$\liegeblkn$}}}}

\newcommand{\longpr}{
\mathrel{\offinterlineskip
\hskip0ex\hbox{$\shortmid$}\hskip -.8ex\hbox{$\liegebalken$} \hskip -0.75ex\hbox{$\shortmid$}}}

\newcommand{\hochlongpr}{
\mathrel{\vcenter{\offinterlineskip
\vskip -.5ex \hbox{$\longpr$}}}}

\newcommand{\longpair}{
\mathrel{\offinterlineskip
\hskip0ex\hbox{$\shortmid$}\hskip -1.4ex\hbox{$\hochlongpr$} \hskip -1.4ex\hbox{$\shortmid$}}}

\newcommand{\legpart}{
\mathrel{\offinterlineskip
\hskip0ex\hbox{$\shortmid$}\hskip +.8ex\hbox{$\shortmid$} \hskip -2.5ex\hbox{$\longpair$}}}

\usepackage{calc}
\newcounter{PartitionDepth}
\newcounter{PartitionLength}

\newcommand{\partii}[3]{
 \begin{picture}(#3,#1)
 \setcounter{PartitionLength}{#3-#2}
 \setcounter{PartitionDepth}{-1-#1}
 \put(#2,\thePartitionDepth){\line(0,1){#1}}     
 \put(#3,\thePartitionDepth){\line(0,1){#1}}
 \put(#2,\thePartitionDepth){\line(1,0){\thePartitionLength}}
 \end{picture}}
\newcommand{\partiii}[4]{
 \begin{picture}(#4,#1)
 \setcounter{PartitionLength}{#4-#2}
 \setcounter{PartitionDepth}{-1-#1}
 \put(#2,\thePartitionDepth){\line(0,1){#1}}
 \put(#3,\thePartitionDepth){\line(0,1){#1}}
 \put(#4,\thePartitionDepth){\line(0,1){#1}}
 \put(#2,\thePartitionDepth){\line(1,0){\thePartitionLength}} 
 \end{picture}}

\newcommand{\upparti}[2]{
 \begin{picture}(#2,#1)
 \setcounter{PartitionDepth}{#1}
 \put(#2,0){\line(0,1){#1}}
 \end{picture}}
\newcommand{\uppartii}[3]{
 \begin{picture}(#3,#1)
 \setcounter{PartitionLength}{#3-#2}
 \setcounter{PartitionDepth}{#1}
 \put(#2,0){\line(0,1){#1}}     
 \put(#3,0){\line(0,1){#1}}
 \put(#2,\thePartitionDepth){\line(1,0){\thePartitionLength}}
 \end{picture}}
\newcommand{\uppartiii}[4]{
 \begin{picture}(#4,#1)
 \setcounter{PartitionLength}{#4-#2}
 \setcounter{PartitionDepth}{#1}
 \put(#2,0){\line(0,1){#1}}
 \put(#3,0){\line(0,1){#1}}
 \put(#4,0){\line(0,1){#1}}
 \put(#2,\thePartitionDepth){\line(1,0){\thePartitionLength}} 
 \end{picture}}
\newcommand{\uppartiv}[5]{
 \begin{picture}(#5,#1)
 \setcounter{PartitionLength}{#5-#2}
 \setcounter{PartitionDepth}{#1}
 \put(#2,0){\line(0,1){#1}}
 \put(#3,0){\line(0,1){#1}}
 \put(#4,0){\line(0,1){#1}}
 \put(#5,0){\line(0,1){#1}}
 \put(#2,\thePartitionDepth){\line(1,0){\thePartitionLength}} 
 \end{picture}}

\newcommand{\diag}{\ensuremath{\mathrm{diag}}}

\begin{document}

\author{Sven Raum and Moritz Weber}

\begin{center}
{\LARGE \bf {A Connection between Easy Quantum Groups,\\ Varieties of Groups and Reflection Groups}}

\bigskip

{\sc by Sven Raum$^{(1)}$ and Moritz Weber
  \setcounter{footnote}{1}
  \footnotetext{Supported by KU Leuven BOF research grant OT/08/032}
}
\end{center}

\begin{abstract}
We present a link between easy quantum groups, discrete groups and combinatorics.  By this, we infer new connections between quantum isometry groups, reflection groups, varieties of groups and the combinatorics of partitions.  More precisely, we consider easy quantum groups and find a relation to subgroups of the infinite free product $\mathbb Z_2^{*\infty}$ of $\mathbb Z_2=\mathbb Z/2\mathbb Z$.  We obtain a link with reflection groups and thus with varieties of groups, which yields a statement on the complexity of the class of easy quantum groups on the one hand, and a ``quantum invariant'' for varieties of groups on the other hand.  Moreover, we reveal a triangular relationship between easy quantum groups, categories of partitions and discrete groups (reflection groups).  As a by-product, we obtain a large number of new quantum isometry groups.
\end{abstract}

\section*{Introduction}

Easy quantum groups were first considered by Banica and Speicher in 2009 \cite{banicaspeicher09}.  This subclass of the class of compact matrix quantum groups \cite{woronowicz87} is the object of our investigations.  By Woronowicz \cite{woronowicz88}, their representations form a concrete tensor \Cstar-category, and by the work of Banica and Speicher, it can be described by the combinatorics of categories of partitions.  We construct an injection from categories of partitions to subgroups of the infinite free product $\mathbb Z_2^{*\infty}$ of $\mathbb Z_2=\mathbb Z/2\mathbb Z$. From this, we may deduce several relations of easy quantum groups to other fields of mathematics.

Firstly, we find quantum groups associated with varieties of groups \cite{neumann67}.  A variety of groups is the class of all groups that satisfy a given set of identical relations (see Section \ref{sec:introduction-to-varieties}).
We obtain a triangular relation of the form:

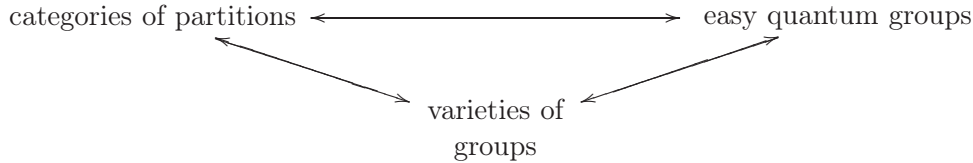
\hspace{1.5cm}
\begin{xy}
\xymatrix{
  \begin{minipage}{4cm}
    \begin{center}
      categories of partitions
    \end{center}
  \end{minipage}
  \ar[dr]
  \ar[rr]
  & &
  \begin{minipage}{4cm}
    \begin{center}
      easy quantum groups
    \end{center}
  \end{minipage}
  \ar[ll]
  \ar[dl]
  \\ &
  \begin{minipage}{3cm}
    \begin{center}
      varieties of groups
    \end{center}
  \end{minipage}
  \ar[ur]
  \ar[ul]
}
\end{xy}

As there are uncountably many varieties of groups \cite{olshanskii70}, we infer that there are uncountably many easy quantum groups. This is a contribution to the classification of easy quantum groups.

Secondly, we obtain a link to quantum isometry groups. A quantum isometry group is the maximal quantum group acting faithfully by isometries on a non-commutative space.  It is the non-commutative replacement of the isometry group.  Quantum isometry groups were studied by Bichon \cite{bichon03}, Banica \cite{banica05-homogeneous-graphs, banica05-small-metric-spaces}, Goswami \cite{goswami09}, Bhowmick and Goswami \cite{bhowmickgoswami09-computations, bhowmickgoswami09-riemannian} and Banica and Skalski \cite{banicaskalski11}.  Banica and Skalski first studied the quantum isometry groups of discrete group duals in \cite{banicaskalski12-group-duals}.  Other examples of such quantum isometry groups where studied by Liszka-Dalecki and So{\l}tan \cite{liszka-daleckisoltan12} and Tao and Qiu \cite{taoqiu12-dihedral-group}.  Notably,  Banica and Skalski related in \cite{banicaskalski11-two-parameter} quantum isometry groups and easy quantum groups for the first time.  We reveal the following triangular relation:

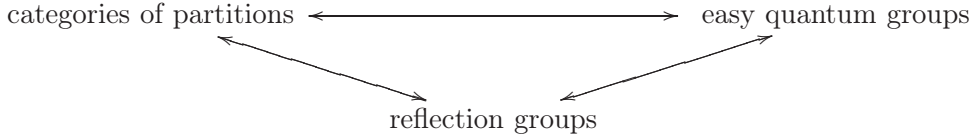
\hspace{1.5cm}
\begin{xy}
\xymatrix{
  \begin{minipage}{4cm}
    \begin{center}
      categories of partitions
    \end{center}
  \end{minipage}
  \ar[dr]
  \ar[rr]
  & &
  \begin{minipage}{4cm}
    \begin{center}
      easy quantum groups
    \end{center}
  \end{minipage}
  \ar[ll]
  \ar[dl]
  \\ &
  \begin{minipage}{3cm}
    \begin{center}
      reflection groups
    \end{center}
  \end{minipage}
  \ar[ur]
  \ar[ul]
}
\end{xy}

The map from reflection groups to easy quantum groups is basically given by the quantum isometry group construction of group duals.

Let us now give an introduction to our methodology.
Next to $q$-deformations of semisimple Lie groups \cite{jimbo85,drinfeld86,woronowicz87,rosso90-version-duke} and quantum isometry groups, quantum groups may arise in the context of the liberation of groups (see \cite{banicaspeicher09} for a concise introduction). The idea is the following.

Let $G$ be a (classical) orthogonal Lie group and consider the \Cstar-algebra $\rC(G)$ of continuous functions on $G$. By means of the fundamental representation $u \in \rC(G) \ot \Cmat{n}$, we can view $\rC(G)$ as a universal $\Cstar$-algebra:
\[\rC(G) \cong \Cstar\left(u_{ij}, 1 \leq i,j \leq n \amid \text{the matrices }(u_{ij}) \text{ and }  (u_{ij}^*) \text{ are unitaries}, u_{ij}u_{kl}=u_{kl}u_{ij}, (R_G) \right) \eqcomma \]
where $(R_G)$ are some further relations of the generators $u_{ij}$.  The liberation $G^+$ of $G$ is a compact matrix quantum group given by the universal \Cstar-algebra
\[\rC(G^+) = \Cstar\left(u_{ij}, 1 \leq i,j \leq n \amid \text{the matrices }(u_{ij}) \text{ and }  (u_{ij}^*) \text{ are unitaries}, (R_G) \right) \eqcomma\]
where we omit the commutativity of the generators $u_{ij}$.

In this sense, Wang \cite{wang95, wang98} constructed three free quantum groups, namely the free orthogonal, the free unitary, and the free symmetric quantum group, liberating the groups $O_n$, $U_n$ and $S_n$. A further example is the hyperoctahedral quantum group $H_n^+$ given by Banica, Bichon, and Collins \cite{banicabichoncollins07}.

The intertwiner spaces of $S_n$, $S_n^+$, $O_n$ and $O_n^+$  admit a combinatorial description by means of partitions.  The process of liberation is reflected by restricting to those partitions that are non-crossing.
In their 2009 article \cite{banicaspeicher09}, Banica and Speicher initiated a systematic study of easy quantum groups, i.e. of those compact matrix quantum groups whose intertwiner spaces are described by the combinatorics of categories of partitions (see Definition 6.3 of \cite{banicaspeicher09} or Definition 1.4 of \cite{weber12}). 
This class of quantum groups includes $S_n^+$, $O_n^+$, and $H_n^+$, as well as the groups $S_n$, $O_n$ and $H_n$, but it goes far beyond the question of liberation of groups. Roughly speaking, it contains all compact quantum groups $G$ with $S_n\subset G\subset O_n^+$, whose intertwiner spaces ``have a nice combinatorics''. It is a consequence of the seminal work by Woronowicz \cite{woronowicz88} that the correspondence between easy quantum groups and their categories of partitions is one-to-one.

The work on easy quantum groups has been continued by Banica, Bichon, Curran, Skalski, So{\l}tan, Speicher, Vergnioux, and the authors of the present article in a couple of articles \cite{banicavergnioux09, banicavergnioux09_2,banicacurranspeicher09_2, banicacurrranspeicher09, raum10, banicacurran10-gram-matrix, banicabichoncollinscurran11, banicaskalskisoltan12-homoeneous-spaces}.  It has become a useful link between quantum groups, combinatorics and free probability theory \cite{koestlerspeicher09-de-finetti, curranspeicher11-infinitesimal-freeness, curranspeicher11-quantum-invariant-families, banicacurranspeicher12-finetti}.  At the same time, easy quantum groups give rise to interesting operator algebras \cite{vaesvergnioux05,brannan11,freslon12,isono12}. 

The approach of Banica and Speicher constitutes a constructive view on the liberation of groups, and it is still a source of a large number of new examples of compact quantum groups. Amongst others, it lead to the discovery of further examples of free quantum groups (see Theorem 3.16 of \cite{banicaspeicher09} and Corollary 2.10 of \cite{weber12}). These free easy quantum groups (also called free orthogonal quantum groups) and likewise the easy groups were completely classified by Banica and Speicher \cite{banicaspeicher09}, and by the second author \cite{weber12}. Furthermore, examples of half-liberated easy quantum groups were given by Banica, Curran, Speicher, and the second author \cite{banicaspeicher09, banicacurranspeicher09_2, weber12}, and they were completely classified in \cite{weber12}. The half-liberation is given by replacing the commutation relation
\[u_{ij}u_{kl}=u_{kl}u_{ij}\]
by
\[u_{ij}u_{kl}u_{rs}=u_{rs}u_{kl}u_{ij} \eqcomma\]
which can be interpreted as a slight weakening of commutativity.

It quickly turned out, that there are even more easy quantum groups than the above mentioned -- and in the present article, we show that there are in fact uncountably many. 
While the classification of non-hyperoctahedral easy quantum groups is complete \cite{banicacurranspeicher09_2,weber12}, the case of hyperoctahedral easy quantum groups is still open. Hyperoctahedral easy quantum groups are quantum subgroups of $H_n^+$ corresponding to hyperoctahedral categories of partitions, i.e. categories which contain the four block partition $\vierpart$ (four points, which are all connected) but not the double singleton $\singleton\otimes\singleton$ (two points, which are not connected). See Section \ref{sec:categories-of-partitions} for details on partitions and categories.

We isolate a large class of hyperoctahedral easy quantum groups -- which we call \emph{simplifiable} -- with the commutation relations
\[u_{ij}^2u_{kl}=u_{kl}u_{ij}^2 \eqstop\]
By this, the focus is put onto a quite unexplored class of partitions, and new questions arise.  The main feature of these quantum groups is that the squares of the generators $u_{ij}$ commute, whereas the elements $u_{ij}$ itself behave rather like free elements.  This mixture of commutative and non-commutative structures could play a special role in the understanding of non-commutative distributions.
We show that this class is very rich, since it inherits from the correspondence with varieties of groups an interesting structure.  In particular, there are uncountably many simplifiable hyperoctahedral easy quantum groups.

The technical heart of this article is worked out in Sections \ref{sec:simplifiable-categories} and \ref{sec:group-theoretic-framework}, where we construct a map from simplifiable hyperoctahedral categories of partitions to subgroups of $\mathbb Z_2^{*\infty}$.  Given such a category $\mathcal C$, we label the partitions in $\mathcal C$ according to their block structure by letters $a_1,a_2,\dotsc$ in order to obtain words.  Mapping these words to $\mathbb Z_2^{*\infty}$ (where now $a_i^2=e$), we obtain the following main result:

\begin{thmstar}[See Theorem \ref{thm:F-is-1-1}]
\label{thm:intro-technical-result}
 There is a lattice isomorphism between simplifiable hyperoctahedral categories of partitions and proper $S_0$-invariant subgroups of $E$, 
 where $E$ is the subgroup of $\mathbb Z_2^{*\infty}$ consisting of all words of even length.
\end{thmstar}

Here, $S_0$ is the subsemigroup of $\textnormal{End}(\mathbb Z_2^{*\infty})$ generated by all inner automorphisms and by finite identifications of letters.
This way, we obtain a one-to-one correspondence with a class of invariant subgroups of $\F_\infty$, which contains the lattice of fully characteristic subgroups of $\F_\infty$.  This lattice in turn is anti-isomorphic to the lattice of varieties of groups \cite{neumann67}.  See Section \ref{sec:introduction-to-varieties} for an introduction to varieties of groups and fully characteristic subgroups. 
By Olshanskii \cite{olshanskii70}, there are uncountably many varieties of groups. Hence, we derive the following theorem.

\begin{thmstar}[See Theorems \ref{thm:easy-quantum-groups-and-varieties}, and \ref{thm:uncountably-many-easy-quantum-groups}]
\label{thm:intro-varieties-qgs}
  There is an injection of lattices of varieties of groups into the lattice of easy quantum groups.  In particular, there are uncountably many easy quantum groups that are pairwise non-isomorphic.
\end{thmstar}

The correspondence between categories of partitions, reflection groups and easy quantum groups is explained in Section \ref{sec:triangular-relationship}.  It is based on the following result.  If $H$ is an $S_0$-invariant subgroup of $\ZZ_2^{* \infty}$, denote by $(H)_n$ the set of all words in $H$ that involve at most the first $n$ letters of $\ZZ_2^{\infty}$.  Denote by $H_n^{[\infty]}$ the maximal simplifiable hyperoctahedral easy quantum group.

\begin{thmstar}[See Theorems \ref{thm:F-and-diagonal-subgroup} and \ref{thm:qiso-groups}]
\label{thm:intro-reflection-groups-qgs}
  If $H \leq E \leq \ZZ_2^{* \infty}$ is a proper $S_0$-invariant subgroup of $E$, then
  \[
    H_n^{[\infty]} \cap \mathrm{QISO}(\Cstar(\ZZ_2^{*n}/(H)_n))
  \]
  is a simplifiable hyperoctahedral easy quantum group.

  Vice versa, the diagonal subgroup of any simplifiable hyperoctahedral easy quantum group is of the form $\ZZ_2^{* n}/(H)_n$ for some proper $S_0$-invariant subgroup $H \leq E$.  Moreover, these two operations are inverse to each other.
\end{thmstar}

This correspondence in connection with Theorem \ref{thm:intro-varieties-qgs}, yields a large class of examples of non-classical quantum isometry groups.


\subsection*{Acknowledgements}
\label{sec:acknowledgements}

The first author thanks Roland Speicher for inviting him to Saarbr\"ucken, where this work was initiated in June 2012.  Both authors thank Stephen Curran for pointing out an error at an early stage of this work.  We are very grateful to Teodor Banica and Adam Skalski for useful discussions on an earlier version of this paper.  We proved the results of Section \ref{sec:realation-qiso-groups} only after Teodor Banica asked whether there is a direct relationship between discrete groups and simplifiable hyperoctahedral easy quantum groups.


\section{Preliminaries and notations}
\label{sec:preliminaries}

In the whole paper, tensor products of \Cstar-algebras are taken with respect to the minimal \Cstar-norm.

\subsection{Compact quantum groups and compact matrix quantum groups}
\label{sec:compact-quantum-groups}

In \cite{woronowicz98}, Woronowicz defines a \emph{compact quantum group} (CQG) as a unital \Cstar-algebra $A$ with a unital *-homomorphism $\Delta: A \ra A \ot A$ such that
\begin{itemize}
\item $\Delta$ is \emph{coassociative}, i.e. $(\Delta \ot \id) \circ \Delta = (\id \ot \Delta) \circ \Delta$,
\item $(A,\Delta)$ is \emph{bisimplifiable}, i.e. the subspaces $\lspan \Delta(A)(1 \ot A)$ and $\lspan \Delta(A)(A \ot 1)$ are dense in $A \ot A$.
\end{itemize}
If $(A, \Delta)$ is a CQG, then $\Delta$ is called its \emph{comultiplication}. Note that the bisimplifiability condition is in fact an assumption on left and right cancellation  (see \cite[Remark 3]{woronowicz98} or \cite[Proof of Proposition 5.1.3]{timmermann08}).  A morphism between two CQGs $A$ and $B$ is a unital *-homomorphism $\phi: A \ra B$ such that $(\phi \ot \phi) \circ \Delta_A = \Delta_B \circ \phi$.  We say that $A$ is a \emph{quantum subgroup} of $B$ if there is a surjective morphism $B \thra A$, and they are \emph{isomorphic} if there is a bijective morphism between them.

A \emph{unitary corepresentation matrix} of $A$ is a unitary element $u \in \rM_n(A)$ such that ${\Delta_A(u_{ij}) = \sum_k u_{ik} \ot u_{kj}}$ for all $1 \leq i,j \leq n$.

The concept of CQGs evolved from \emph{compact matrix quantum group} (CMQG), \cite{woronowicz87,woronowicz91}.  A compact matrix quantum group is a unital \Cstar-algebra $A$ with an element $u \in \rM_n(A)$ such that
\begin{itemize}
\item $A$ is generated by the entries of $u$,
\item there is a *-homomorphism $\Delta: A \ra A \ot A$ such that $\Delta(u_{ij}) = \sum_k u_{ik} \ot u_{kj}$ for all $1 \leq i,j \leq n$,
\item $u$ and its transpose $u^t$ are invertible.
\end{itemize}
Every CMQG gives rise to a CQG, but the former contains more information -- the choice of $u$.  The matrix $u$ is called the \emph{fundamental corepresentation} of $(A, u)$ and it is a corepresentation matrix of $(A, \Delta)$.  A morphism between CMQGs $A$ and $B$ is a morphism of the underlying CQGs such that $(\phi \ot \id)(u_A)$ is conjugate by a matrix in $\GL{n}{\CC}$ with $u_B$.  If $A$ and $B$ are CMQGs and there is a bijective morphism of CMQGs between them, we say that they are \emph{similar}.  We say that two CMQG are \emph{isomorphic} if they are isomorphic as CQGs.

\subsection{Tannaka-Krein duality, easy quantum groups and categories of partitions}
\label{sec:easy-quantum-groups-partition}

\subsubsection{Woronowicz' Tannaka-Krein duality}
\label{sec:tannaka-krein-duality}

If $(A, \Delta)$ is a CQG and $u \in \bo(H) \ot A$ for some Hilbert space $H$, then $u$ is a \emph{unitary corepresentation} of $A$ if $u$ is a unitary and $(\id \ot \Delta)(u) = u_{12}u_{13}$.  We used the leg notation: ${u_{12} = u \ot 1}$, ${u_{13} = (\id \ot \Sigma)(u \ot 1)}$, where $\Sigma$ is the flip on $A \ot A$.  A morphism between unitary corepresentations $u \in \bo(H) \ot A$ and $v \in \bo(K) \ot A$ is a bounded linear operator $T \in \bo(H, K)$ such that $(T \ot 1) \circ u = v \circ (T \ot 1)$.  A morphism between two unitary corepresentations is also called an \emph{intertwiner}.  The space of intertwiners between two unitary corepresentations $u \in A \ot \bo(H)$ and $v \in A \ot \bo(K)$ is denoted by $\Hom(u, v)$.
With this structure, the finite dimensional unitary corepresentations of a CQG $(A, \Delta)$ form the  \emph{concrete \Cstar-category} $\Corepfin(A)$, i.e. a \Cstar-category with a faithful \Cstar-functor $\Corepfin \ra \FdHilb$ to the category of finite dimensional Hilbert spaces (see \cite{woronowicz88} for details).  The \emph{tensor product} of two corepresentation $u \in \bo(H) \ot A$ and $v \in \bo(K) \ot A$ is defined by $u \ot v = u_{13}v_{23}$.  This tensor product induces the structure of a \emph{concrete complete compact tensor \Cstar-category} in the sense of Woronowicz on $\Corepfin(A)$ (see \cite{woronowicz87,woronowicz88} or \cite[Chapter 5]{timmermann08}).

The fundamental corepresentation of a CMQG is a generator of its category of finite dimensional corepresentations.
\begin{theorem}[See Proposition 6.1.6 of \cite{timmermann08}]
\label{thm:fundamental-corepresentation-generates}
  If $v$ is a unitary corepresentation of a compact matrix quantum group $(A, u)$, then there is $k \in \NN$ such that $v$ is a subobject of $u^{\ot k}$.
\end{theorem}

Woronowicz proved the following version of Tannaka-Krein duality.
\begin{theorem}[See \cite{woronowicz88}]
\label{thm:tannaka-krein-duality}
 Every concrete complete compact tensor \Cstar-category is the category of corepresentations of a compact quantum group.  Two compact quantum groups $A$ and $B$ are isomorphic if and only if their categories of corepresentations are equivalent over $\FdHilb$.  If $A$ and $B$ are compact matrix quantum groups, they are similar if and only if their categories of corepresentations are equivalent over $\FdHilb$ by a functor preserving the isomorphism class of the fundamental corepresentation.
\end{theorem}

\subsubsection{Categories of partitions}
\label{sec:categories-of-partitions}

In order to describe corepresentation categories of quantum groups combinatorially, Banica and Speicher introduced the notions of a \emph{category of partitions} and of \emph{easy quantum groups} \cite{banicaspeicher09}.  A \emph{partition} $p$ is given by $k$ upper points and $l$ lower points which may be connected by lines. By this, the set of $k+l$ points is partitioned into several \emph{blocks}. We write a partition as a diagram in the following way:
\setlength{\unitlength}{0.5cm}
\begin{center}
  \begin{picture}(14,3)
    \put(0,0){$\cdot$}
    \put(1,0){$\cdot$}
    \put(2,0){$\cdot$}
    \put(3,0){$\cdot$}
    \put(4.5,0){$\dotsc$}
    \put(6,0){$\cdot$}

    \put(3,1.5){$p$}

    \put(0,3){$\cdot$}
    \put(1,3){$\cdot$}
    \put(2,3){$\cdot$}
    \put(3,3){$\cdot$}
    \put(4.5,3){$\dotsc$}
    \put(6,3){$\cdot$}
    
    \put(8,1.5){\begin{minipage}{3.5cm} $k$ upper points and \\ $l$ lower points. \end{minipage}}
  \end{picture}
\end{center}

Two examples of such partitions are the following diagrams.
\setlength{\unitlength}{0.5cm}
\begin{center}
  \begin{picture}(6,3)
    \put(0,0){\line(0,1){2}}
    \put(1,0){\line(0,1){2}}
    \put(0,1){\line(1,0){1}}

    \put(4,0){\line(0,1){0.5}}
    \put(5,0){\line(0,1){1}}
    \put(4,2){\line(0,-1){1}}
    \put(5,2){\line(0,-1){0.5}}
    \put(4,1){\line(1,0){1}}
    
    \put(5.2,0){.}
  \end{picture}
\end{center}
In the first example, all four points are connected, and the partition consists only of one block.  In the second example, the left upper point and the right lower point are connected, whereas neither of the two remaining points is connected to any other point.

The set of partitions on $k$ upper and $l$ lower points is denoted by $P(k,l)$, and the set of all partitions is denoted by $P$.
A partition $p \in P(k,l)$ is called \emph{noncrossing}, if it can be drawn in such a way that none of its lines cross.

A few partitions play a special role in this article, and they are listed here:
\begin{itemize}
 \item The \emph{singleton partition} $\singleton$ is the partition in $P(0,1)$ on a single lower point.
 \item The \emph{double singleton partition} $\singleton\otimes\singleton$ is the partition in $P(0,2)$ on two non-connected lower points.
 \item The \emph{pair partition} (also called \emph{duality partition}) $\paarpart$ is the partition in $P(0,2)$ on two connected lower points.
 \item The \emph{unit partition} (also called \emph{identity partition}) $\idpart$ is the partition in $P(1,1)$ connecting one upper with one lower point.
 \item The \emph{four block partition} $\vierpart$ is the partition in $P(0,4)$ connecting four lower points.
 \item The \emph{s-mixing partition} $h_s$ is the partition in $P(0,2s)$ for $s\in \N$ given by two blocks connecting the $2s$ points in an alternating way:
\setlength{\unitlength}{0.5cm}
\begin{center}
\begin{picture}(13,4)
\put(-1,1.5){$h_s\;=$}
\put(-0.1,1){\uppartiii{2}{1}{3}{5}}
\put(5,3){\line(1,0){1}}
\put(6.75,2.95){$\ldots$}
\put(-0.1,1){\uppartiii{1}{2}{4}{6}}
\put(6,2){\line(1,0){1}}
\put(7.75,1.95){$\ldots$}
\put(-0.1,1){\uppartii{1}{10}{12}}
\put(8.65,3){\line(1,0){0.5}}
\put(-0.1,1){\uppartii{2}{9}{11}}
\put(9.5,2){\line(1,0){1}}
\put(12.25,1){.}
\end{picture}
\end{center}
 \item The \emph{crossing partition} (also called \emph{symmetry partition}) $\crosspart$ is the partition in $P(2,2)$ connecting the upper left with the lower right point, as well as the upper right point with the lower left one. It is the partition of two crossing pair partitions.
 \item The \emph{half-liberating partition} $\halflibpart$ is the partition in $P(3,3)$ given by the blocks  $\{1,3'\}$, $\{2,2'\}$ and $\{3,1'\}$ connecting three upper points $1,2,3$ and three lower points $1',2',3'$ such that $1$ and $3'$ are connected, $2$ and $2'$, and finally $3$ and $1'$.
\end{itemize}
Further partitions will be introduced in Section \ref{sec:genericity-of-primariy-partition}.

We will also use \emph{labelled partitions}, i.e. partitions whose points are either labelled by numbers or by letters. The labelling of a partition $p\in P(k,l)$ with letters is usually proceeded by starting at the very left of the $k$ upper points of $p$ and then going clockwise, ending at the very left of the $l$ lower points. The labelling with numbers typically labels both the upper and the lower row of points from left to right.

There are the natural operations \emph{tensor product} ($p\otimes q$), \emph{composition} ($pq$), \emph{involution} ($p^*$) and \emph{rotation} on partitions (see \cite[Definition 1.8]{banicaspeicher09} or \cite[Definition 1.4]{weber12}).  A collection $\cC$ of subsets $D(k,l) \subset P(k,l)$, $k,l \in \NN$ is called a \emph{category of partitions} if it is closed under these operations and if it contains the pair partition $\paarpart$, and the unit partition $\idpart$ (see \cite[Definition 6.1]{banicaspeicher09} or \cite[Definition 1.4]{weber12}).

A category of partitions $\cC$ is called \emph{hyperoctahedral} if the four block $\vierpart$ is in $\cC$, but the double singleton $\singleton \ot \singleton$ is not in $\cC$.  

Given a partition $p \in P(k,l)$ and two multi-indices $(i_1, \dotsc, i_k)$, $(j_1, \dotsc, j_l)$, we can label the diagram of $p$ with these numbers (now, the upper and the lower row both are labelled from left to right, respectively) and we put
\[\delta_p(i,j)
  =
  \begin{cases}
    1 & \text{if } p \text{ connects only equal indices,} \\
    0 & \text{if there is a string of } p \text{ connecting unequal indices} \eqstop
  \end{cases}\]
For every $n \in \NN$, there is a map $T_p: (\CC^n)^{\ot k} \ra (\CC^n)^{\ot l}$  associated with $p$, which is given by
\[T_p(e_{i_1} \ot \dotsm \ot e_{i_k}) = \sum_{1 \leq j_1, \dotsc, j_l \leq n} \delta_p(i, j) \cdot e_{j_1} \ot \dotsm \ot e_{j_l} \eqstop \]

\begin{definition}[Definition 6.1 of \cite{banicaspeicher09} or Definition 2.1 of \cite{banicacurranspeicher09_2}]
  A compact matrix quantum group $(A,u)$ is called \emph{easy}, if there is a category of partitions $\cC$ given by $D(k,l) \subset P(k,l)$, for all $k,l \in \NN$ such that
  \[\Hom(u^{\ot k}, u^{\ot l}) = \lspan \{T_p \amid p \in D(k,l) \} \eqstop \]
\end{definition}

Combining Theorems \ref{thm:fundamental-corepresentation-generates} and \ref{thm:tannaka-krein-duality}, we obtain the following theorem, which is the basis of all combinatorial investigation on easy quantum groups.
\begin{theorem}[See \cite{banicaspeicher09}]
\label{thm:classification-easy-quantum-groups-by-partitions}
  There is a bijection between categories of partitions and easy quantum groups up to similarity.
\end{theorem}

Thus, easy quantum groups are completely determined by their categories of partitions.

\subsection{Quantum isometry groups}
\label{sec:introduction-qiso-groups}

Given a discrete group $G$ with finite generating set $S \subset G$ and associated word-length function $l: G \ra \NN$, $l(g) = \min \{ n \in \NN \amid \exists s_1, \dotsc, s_n \in S: g = s_1 \dotsm s_n\}$,  we obtain a quantum isometry group of $\Cstarmax(G)$ along the lines of \cite{banicaskalski11}.  We denote by $u_g$ the canonical unitary of $\Cstarmax(G)$ associated with $g \in G$.

\begin{definition}[Definitions 2.5  and Section 4 of \cite{banicaskalski11}]
  Let $(A, u = (u_{st})_{s,t \in S})$ be a compact matrix quantum group and write $\mathrm{Pol}(A)$ for its polynomial algebra $*-\mathrm{alg}(u_{st} \amid s,t \in S)$.  An action $\alpha: \Cstarmax(G) \ra \Cstarmax(G) \ot A$ on $\Cstarmax(G)$ is faithful and isometric with respect to $l$, if
  \begin{itemize}
    \item  $\alpha(L_n) \subset L_n \ot \mathrm{Pol}(A)$ for all $n \in \NN$, where $L_n = \lspan\{u_g \amid l(g) = n\}$ and
    \item $\alpha(u_s) = \sum_{t \in S} u_t \ot u_{ts}$ for all $s \in S$.
  \end{itemize}

\end{definition}

\begin{theorem}[Theorems 2.7 and 4.5 of \cite{banicaskalski11}]
  There is a maximal compact matrix quantum group $(A,u)$ acting faithfully and isometrically with respect to $l$ on $\Cstarmax(G)$.  That is, for any other compact matrix quantum group $(B, v)$ acting faithfully and isometrically with respect to $l$ on $\Cstarmax(G)$ there is a unique morphism of CMQGs $\phi:(A, u) \ra (B,v)$ such that $\phi(u) = v$.
\end{theorem}

\subsection{Varieties of groups}
\label{sec:introduction-to-varieties}

In this section we briefly explain the concepts of varieties of groups.  We advice the interested reader to consult \cite{neumann67} for a thorough introduction.

Consider $\freegrp{\infty}$ with free basis $x_1, x_2, \dotsc $ and let $w \in \freegrp{\infty}$ be a word in the letters $x_1, x_2, \dotsc x_n$.  We say that the \emph{identical relation} $w$ holds in a group $G$ if for any choice of elements $g_1,g_2, \dotsc g_n \in G$, replacing $x_i$ by $g_i$, we have $w(g_1, \dotsc, g_n) = 1_G$.  Following \cite{neumann67} a \emph{variety of groups} $\cV$ is a class of groups for which there is a set of words $R \subset \freegrp{\infty}$ such that every group $G$ in $\cV$ satisfies the identical relations in $R$.

Let us give some examples of varieties of groups.

\begin{example}
\label{ex:varieties}
The following classes of groups are varieties of groups.  We also describe the identical relations that characterise them.
  \begin{enumerate}
  \item The class of all groups is the variety of groups, where no law is satisfied.
  \item The class of abelian groups is defined by the commutator $[x,y] = xyx^{-1}y^{-1}$.
  \item The class of groups with a fixed exponent $s$ is given by $x^s$.
  \item The class of nilpotent groups of class $2$ is described by $[[x,y],z]$.
  \end{enumerate}
\end{example}

Varieties of groups are important for this work, because they correspond precisely to the \emph{fully characteristic subgroups} of $\freegrp{\infty}$.  Given an inclusion of groups $H \leq G$, $H$ is fully characteristic in $G$, if it is invariant under all endomorphisms of $G$.  This means that $\phi(H) \subset H$ for every endomorphism $\phi \in \End(G)$.

 The set of identical relations that hold in a given group, form a subgroup of $\freegrp{\infty}$.  This observation is the trigger to prove the following theorem.

\begin{theorem}[See \cite{neumann37} or Theorem 14.31 in \cite{neumann67}]
\label{thm:varieties-and-fully-characteristic-subgroups}
  There is a lattice anti-isomorphism between varieties of groups and fully characteristic subgroups of $\freegrp{\infty}$ sending a variety of groups to the set of all identical relations that hold in it.
\end{theorem}

We will make use of another observation concerning elements of free groups. Two sets of words in $\freegrp{n}$ are called \emph{equivalent}, if they generate the same fully characteristic subgroup.

\begin{theorem}[See Theorem 12.12 in \cite{neumann67}]
\label{thm:splitting-of-powers}
  Every word $w \in \freegrp{n}$, $n \in \NN \cup \{\infty\}$ is equivalent to a pair of words $a$ and $b$ in $\freegrp{n}$, where $a$ is of the form $x^m$ for some $m \geq 2$ and $x \in \freegrp{n}$, and $b$ is an element of the commutator subgroup $[\freegrp{n}, \freegrp{n}]$.
\end{theorem}


\section{Simplifiable hyperoctahedral categories}
\label{sec:simplifiable-categories}

\subsection{A short review of the classification of easy quantum groups}
\label{sec:review-classification}

Recall from Section \ref{sec:categories-of-partitions} that a category of partitions is called \emph{hyperoctahedral}, if it contains the four block partition $\vierpart$ but not the partition $\singleton\otimes\singleton$.  An easy quantum group $G$ is called hyperoctahedral, if its corresponding category of partitions is hyperoctahedral.  By \cite[Theorem 6.5]{banicacurranspeicher09_2} and \cite[Corollary 4.11]{weber12} we know that there are exactly 13 nonhyperoctahedral easy quantum groups, resp. 13 nonhyperoctahedral categories of partitions, so they are completely classified.  In this article we will shed some light on the classification of hyperoctahedral categories.  Let us first give a short review of the classification of easy quantum groups.

For partitions $p_1,\ldots,p_n\in P$, we write  $\cC=\langle p_1,\ldots,p_n\rangle$ for the category generated by these partitions, i.e. $\cC$ is the smallest subclass of $P$ which is closed under the category operations (see Section \ref{sec:categories-of-partitions}) and which contains the partitions $p_1,\ldots,p_n$. (Note that the pair partition $\paarpart$ and the unit partition $\idpart$ are always contained in a category as trivial base cases.)

By \cite[Theorem 3.16]{banicaspeicher09} and \cite[Corollary 2.10]{weber12}, there are exactly seven free easy quantum groups (also called free orthogonal quantum groups), namely:

\begin{align*}
 B_n^+ &\qquad\subset &{B_n'}^{\!+} &\qquad\subset &B_n^{\#+} &\qquad\subset &O_n^+\\
\upsubset\; & &\upsubset\; & & & &\upsubset\;\\
 S_n^+ &\qquad\subset &{S_n'}^{\!+} & &\subset &  &H_n^+ &.
\end{align*}

The corresponding seven categories of partitions are described as follows.

\begin{align*}
 \langle\singleton\rangle &\qquad\supset &\langle\legpart\rangle &\qquad\supset &\langle\singleton\otimes\singleton\rangle &\qquad\supset &\langle\emptyset\rangle=NC_2\\
\downsubset\; & &\downsubset\; & & & &\downsubset\;\\
 \langle\singleton, \vierpart\rangle=NC &\qquad\supset &\langle\singleton\otimes\singleton, \vierpart\rangle & &\supset &  &\langle\vierpart\rangle & .
\end{align*}

Note that these partitions are all noncrossing, i.e. all of these seven categories are subclasses of $NC$, the collection of all noncrossing partitions.  We denote by $NC_2$ the category of all noncrossing pair partitions.  Furthermore, note that only $\langle\vierpart\rangle$ is a hyperoctahedral category, the category corresponding to the hyperoctahedral quantum group $H_n^+$ by \cite{banicabichoncollins07}. The other six categories are nonhyperoctahedral.

Besides the noncrossing categories, there are many categories which contain partitions that have some crossing lines. The most prominent partition which involves a crossing is the \emph{crossing partition} (also called \emph{symmetry partition}) $\crosspart$ in $P(2,2)$. Every category containing the crossing partition corresponds to a group. By \cite[Theorem 2.8]{banicaspeicher09} we know that there are exactly six easy groups.

\begin{align*}
 B_n &\qquad\subset &B_n'  &\qquad\subset &O_n\\
\upsubset\; & &\upsubset\; &  &\upsubset\;\\
 S_n &\qquad\subset &S_n' &\qquad\subset  &H_n&.
\end{align*}

Accordingly, there are exactly six categories of partitions containing the crossing partition $\crosspart$.

\begin{align*}
 \langle\crosspart,\singleton\rangle  &\qquad\supset &\langle\crosspart,\singleton\otimes\singleton\rangle &\qquad\supset &\langle\crosspart\rangle=P_2\\
\downsubset\; & &\downsubset\; &  &\downsubset\;\\
 \langle\crosspart,\singleton, \vierpart\rangle=P &\qquad\supset &\langle\crosspart,\singleton\otimes\singleton, \vierpart\rangle &\qquad\supset  &\langle\crosspart,\vierpart\rangle &.
\end{align*}

Note that on the level of categories containing the crossing partition, the two categories $\langle\crosspart,\legpart\rangle$ and $\langle\crosspart,\singleton\otimes\singleton\rangle$ coincide. Furthermore, amongst the above categories only $\langle\crosspart,\vierpart\rangle$ is hyperoctahedral; the other five categories are nonhyperoctahedral.

\emph{Half-liberated} easy quantum groups were introduced in \cite{banicaspeicher09} and \cite{banicacurranspeicher09_2}. They correspond to categories containing the half-liberating partition $\halflibpart$ but not the crossing partition $\crosspart$. By \cite[Theorem 4.13]{weber12}, there are exactly the following half-liberated easy quantum groups, containing the \emph{hyperoctahedral series} $H_n^{(s)}$, $s\geq 3$ of \cite[Definition 3.1]{banicacurranspeicher09_2}.

\begin{align*}
 B_n^{\#*} &&\qquad\subset &&O_n^*\\
 &&\quad &&\upsubset\;\\
 &&  &&H_n^*\\
 && &&\upsubset\;\\
 &&  &&H_n^{(s)}, s\geq 3&.
\end{align*}

The corresponding  categories of partitions are described as follows.

\begin{align*}
 \langle\halflibpart,\singleton\otimes\singleton\rangle &&\qquad\supset &&\langle\halflibpart\rangle\\
 && &&\downsubset\;\\
 && &&\langle\halflibpart, \vierpart\rangle\\
 && &&\downsubset\;\\
 && &&\langle\halflibpart, \vierpart, h_s\rangle&.
\end{align*}

Here, $\langle\halflibpart, \vierpart\rangle$ and $\langle\halflibpart, \vierpart, h_s\rangle$ are hyperoctahedral, for all $s\geq 3$. The categories  $\langle\halflibpart,\singleton\otimes\singleton\rangle$ and $\langle\halflibpart\rangle$ in turn are two more nonhyperoctahedral categories, completing the list of 13 nonhyperoctahedral categories.

We conclude that the only class of categories which ought to be classified is the one of hyperoctahedral categories, as illustrated by the following picture.

\begin{center}
\begin{picture}(22,9)
  \put(0,8){$\langle \singleton, \vierpart \rangle$}
  \put(4,8){$\supset$}
  \put(6,8){$\langle \singleton \ot \singleton, \vierpart \rangle$}
  \put(16,8){$\supset$}
  \put(18,8){$\langle \vierpart \rangle$}

  \put(18.5,6.25){$\downsubset$}
  \put(18.5,4){\textbf{?}}
  \put(18.5,2.25){$\downsubset$}
  \put(14,2){\begin{rotate}{45} $\supset$ \end{rotate}}

  \put(0.5, 4.25){$\downsubset$}

  \put(-1.5,0){$\langle \crosspart, \singleton, \vierpart \rangle = P$}
  \put(4,0){$\supset$}
  \put(6,0){$\langle \crosspart, \singleton \ot \singleton, \vierpart \rangle$}
  \put(16,0){$\supset$}
  \put(18,0){$\langle \crosspart, \vierpart \rangle \eqstop$}
\end{picture}
\end{center}

The question is to find all categories $\cC$ of partitions, which contain the four block $\vierpart$ but not the double singleton $\singleton\otimes\singleton$. Furthermore, we can restrict to those categories which do not contain the half-liberating partition $\halflibpart$. The \emph{higher hyperoctahedral series} $H_n^{[s]}$, $s\in\{3,4,\ldots,\infty\}$ of \cite[Section 4]{banicacurranspeicher09_2} fall into this class. They are given by the categories $\langle \vierpart, h_s\rangle$.

\subsection{Base cases in the class of hyperoctahedral categories}
\label{sec:genericity-of-primariy-partition}

By definition, the category $\langle\vierpart\rangle$ is a natural base case in the class of hyperoctahedral categories, but we will see that also other categories serve as base cases for interesting subclasses of hyperoctahedral categories.
For this, we introduce two more partitions.

\begin{definition}
 The \emph{fat crossing partition} $\fatcrosspart$
 is the following partition in $P(4,4)$, connecting the upper points 1 and 2 with the lower points $3'$ and $4'$, as well as the upper points 3 and 4 with the lower points $1'$ and $2'$, i.e. $\fatcrosspart$ consists of two crossing four blocks.
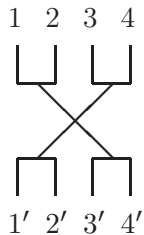
\setlength{\unitlength}{0.5cm}
 \begin{figure}[h]
 \begin{center}
 \begin{picture}(5,7)
 \put(1,5.5){1}
 \put(2,5.5){2}
 \put(3,5.5){3}
 \put(4,5.5){4} 
 \put(1,0){$1'$}
 \put(2,0){$2'$}
 \put(3,0){$3'$}
 \put(4,0){$4'$}
 \thicklines
 \put(0,6){\partii{1}{1}{2}}
 \put(0,6){\partii{1}{3}{4}}
 \put(1.8,2){\line(1,1){2}}
 \put(1.8,4){\line(1,-1){2}}
 \put(0,1){\uppartii{1}{1}{2}}
 \put(0,1){\uppartii{1}{3}{4}}
 \end{picture}
 \end{center}
 \end{figure}
\end{definition}

Note that any category $\mathcal C$ containing the fat crossing, also contains the four block partition (see also Lemma \ref{LemBezBasePart}). If furthermore $\singleton\otimes\singleton\notin\mathcal C$, then $\mathcal C$ is hyperoctahedral. The converse is also true: Any hyperoctahedral category (apart from $\langle\vierpart\rangle$) contains the fat crossing (see Proposition \ref{PropFat}).

\begin{definition}
 The \emph{pair positioner partition} $\primarypart$ is the following partition in $P(3,3)$, consisting of a four block on 1, 2, $2'$ and $3'$ and a pair on 3 and $1'$.
\setlength{\unitlength}{0.5cm}
\begin{center}
  \begin{picture}(5,6)
    \thicklines
    \put(0,6){\partii{1}{1}{2}}
    \put(1.2,1){\line(1,2){2}}
    \put(1.8,4){\line(1,-2){1}}
    \put(0,1){\uppartii{1}{2}{3}}
    \put(1,5.5){1}
    \put(2,5.5){2}
    \put(3,5.5){3}
    \put(1,0){$1'$}
    \put(2,0){$2'$}
    \put(3,0){$3'$}
  \end{picture}
\end{center}
\end{definition}

\begin{remark}
In \cite[Lemma 4.2]{banicacurranspeicher09_2} the following partitions $k_l\in P(l+2,l+2)$ for $l\in \N$ were used to define the \emph{higher hyperoctahedral series} $H_n^{[s]}$ (see also \cite{weber12} for a definition of $k_l$). They are given by a four block on $\{1,1',l+2,(l+2)'\}$ and pairs on $\{i,i'\}$ for $i=2,\ldots,l+1$. The following picture illustrates the partition $k_l$ -- note that the waved line from $1'$ to $l+2$ is \emph{not} connected to the lines from $2$ to $2'$, from $3$ to $3'$ etc.
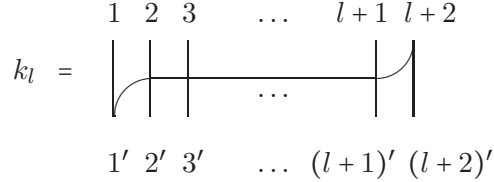
\setlength{\unitlength}{0.5cm}
\begin{center}
  \begin{picture}(14,5)
    \put(0.5,2.5){$k_l\;\;=$}
    \put(1.9,1.5){\upparti{2}{1}}
    \put(1.9,1.5){\upparti{2}{2}}
    \put(1.9,1.5){\upparti{2}{3}}
    \put(7,2){$\ldots$}
    \put(1.9,1.5){\upparti{2}{8}}
    \put(1.9,1.5){\upparti{2}{9}}
    \put(3,0){$1'$}
    \put(4,0){$2'$}
    \put(5,0){$3'$}
    \put(7,0){$\ldots$}
    \put(8.4,0){$(l+1)'$}
    \put(11,0){$(l+2)'$}
    \put(3,4){$1$}
    \put(4,4){$2$}
    \put(5,4){$3$}
    \put(7,4){$\ldots$}
    \put(9.1,4){$l+1$}
    \put(10.9,4){$l+2$}
    
    \put(7.2,1.5){\oval(8,2)[tl]}
    \put(7.15,3.5){\oval(8,2)[br]}
  \end{picture}
\end{center}

We check that $k_1$ is in a category $\mathcal C$ if and only if all $k_l$ are in $\cC$ for all $l\in\N$ (apply the pair partition to $k_l\otimes k_1$ to obtain $k_{l+1}$). The pair positioner partition $\primarypart$ is a rotated version of $k_1$, thus $\primarypart\in\mathcal C$ if and only if $k_1\in\mathcal C$. Furthermore, $\langle\primarypart\rangle$ corresponds to $H_n^{[\infty]}$ of \cite{banicacurranspeicher09_2}.
\end{remark}

The fat crossing partition $\fatcrosspart$ can be constructed out of the pair positioner partition $\primarypart$ using the category operations. The following lemma shows some relations between the partitions.

\begin{lemma}\label{LemBezBasePart}
The following partitions may be generated inside the following categories using the category operations.
\begin{itemize}
  \item[(i)] $\vierpart\in\langle\fatcrosspart\rangle$.
  \item[(ii)] $\fatcrosspart\in\langle\primarypart\rangle$.
  \item[(iii)] $\primarypart\in\langle h_s\rangle$ for all $s\geq 3$.
  \item[(iv)] $\primarypart\in\langle\halflibpart,\vierpart\rangle$.
\end{itemize}
\end{lemma}
\begin{proof}
(i) We obtain $\vierpart$ as the composition of $\fatcrosspart$, $\idpart\otimes\paarpart\otimes\idpart$ and $\paarpart$.

(ii) Compose the tensor product $\vierpart\otimes\vierpart$ with $\primarypart$ in the following way:

\newsavebox{\primary}
\savebox{\primary}
{
  \begin{picture}(5,6)
    \thicklines
    \put(0,6){\partii{1}{1}{2}}
    \put(1.2,1){\line(1,2){2}}
    \put(1.8,4){\line(1,-2){1}}
    \put(0,1){\uppartii{1}{2}{3}}
  \end{picture}
}

\setlength{\unitlength}{0.5cm}
\begin{center}
  \begin{picture}(19,12)
    \put(0,10){\uppartiv{1}{1}{2}{3}{4}}
    \put(0,10){\uppartiv{1}{5}{6}{7}{8}}
    
    \put(0,5.5){\upparti{4}{1}}
    \put(0,5.5){\upparti{4}{2}}
    \put(1.8,4.4){\usebox{\primary}}
    \put(0,5.5){\upparti{4}{6}}
    \put(0,5.5){\upparti{4}{7}}
    \put(0,5.5){\upparti{4}{8}}
    
    \put(0,1){\upparti{4}{1}}
    \put(0,1){\upparti{4}{2}}
    \put(0,1){\upparti{4}{3}}
    \put(2.8,0){\usebox{\primary}}
    \put(0,1){\upparti{4}{7}}
    \put(0,1){\upparti{4}{8}}
    
    \put(10,6){$=$}

    \put(10.5,5.5){\uppartiv{1}{1}{2}{5}{6}}
    \put(10.5,5.5){\uppartiv{2}{3}{4}{7}{8}}

    \put(19,5.5){.}
  \end{picture}
\end{center}
Then use rotation to obtain $\fatcrosspart$.

(iii) We construct the rotated version $k_1$ of $\primarypart$ using $h_s\otimes \idpart^{\otimes 3}$ and its rotated version:
\setlength{\unitlength}{0.5cm}
\begin{center}
  \begin{picture}(21,9)
    \put(0,5){\uppartiii{2}{1}{3}{5}}
    \put(5,7){\line(1,0){1}}
    \put(8.25,6.95){$\ldots$}
    \put(0,5){\uppartiii{1}{2}{4}{6}}
    \put(6,6){\line(1,0){1}}
    \put(8.25,5.95){$\ldots$}
    \put(12,5){\upparti{2}{1}}
    \put(12,5){\upparti{2}{2}}
    \put(12,5){\upparti{2}{3}}
    \put(10.8,7){\line(1,0){.5}}
    \put(10.8,6){\line(1,0){1.5}}
    \put(10,5){\upparti{2}{1}}
    \put(11,5){\upparti{1}{1}}

    \put(3,5){\partiii{1}{1}{3}{5}}
    \put(8,3){\line(1,0){1}}
    \put(10.4,2.95){$\ldots$}
    \put(3,5){\partiii{2}{2}{4}{6}}
    \put(9,2){\line(1,0){1}}
    \put(10.4,1.95){$\ldots$}
    \put(0,2){\upparti{2}{1}}
    \put(0,2){\upparti{2}{2}}
    \put(0,2){\upparti{2}{3}}
    \put(11.8,3){\line(1,0){.5}}
    \put(11.8,2){\line(1,0){1.5}}
    \put(11,5){\partii{2}{2}{4}}
    \put(11,5){\partii{1}{1}{3}}
    
    \put(16.5,4.5){$=$}
    \put(16.8,3.7){\upparti{2}{1}}
    \put(16.8,3.7){\upparti{2}{2}}
    \put(16.8,3.7){\upparti{2}{3}}
    \put(19.1,3.7){\oval(2,2)[tl]}
    \put(19.1,5.7){\oval(2,2)[br]}
    
    \put(1.1,1){a}
    \put(2.1,1){b}
    \put(3.1,1){a}
    \put(13.1,7.5){a}
    \put(14.1,7.5){b}
    \put(15.1,7.5){a}
    
    \put(17.9,2.8){a}
    \put(18.9,2.8){b}
    \put(19.9,2.8){a}
    \put(17.9,6){a}
    \put(18.9,6){b}
    \put(19.9,6){a}
    
    \put(20.25,2.8){.}
  \end{picture}
\end{center}

(iv) The partition $\vierpartrot\in P(2,2)$ is a rotated version of the four block $\vierpart$. Compose $\vierpartrot\otimes\idpart$ with the half-liberated partition $\halflibpart$ to obtain $\primarypart$.
\end{proof}

The pair positioner partition $\primarypart$ plays an important role in the sequel. By the preceding lemma, we see that any category $\cC$ containing the pair positioner partition $\primarypart$ also contains the four block partition $\vierpart$. Thus, these categories form a subclass of the hyperoctahedral categories.

\begin{definition}
 A hyperoctahedral category $\cC$ is called \emph{simplifiable}, if it contains the pair positioner partition $\primarypart$.
\end{definition}

Simplifiable hyperoctahedral categories carry a nice feature -- they can be described by very simplified partitions. This is the content of Lemma \ref{LemThreeRow}. We first prove a lemma on the block structure of partitions in simplifiable hyperoctahedral categories.

\begin{lemma}\label{LemBlockVerbinden}
 Let $\cC$ be any category of partitions, and let $p\in\cC$.
\begin{itemize}
 \item[(a)] If $\cC$ contains the four block partition $\vierpart$, we can connect neighbouring blocks of $p$ inside of $\cC$, i.e. the partition $p'$ obtained from $p$ by connecting two blocks of $p$ which have at least two neighbouring points is again in $\cC$.
 \item[(b)] If $\cC$ contains the pair positioner partition $\primarypart$, we can connect arbitrary blocks of $p$ inside of $\cC$, i.e. the partition $p'$ obtained from $p$ by combining two arbitrary blocks of $p$ is again in $\cC$.
\end{itemize}
\end{lemma}
\begin{proof}
We may assume that $p$ has no upper points, by rotation.

(a) We can compose $p$ with $\idpart^{\ot \alpha} \ot \vierpartrot \ot \idpart^{\ot \beta}$ for suitable $\alpha$ and $\beta$.


(b) By composition, we insert a pair partition $\paarpart$ next to the block $b_1$ of $p$.
 By (a), we can connect it to $b_1$. Using the pair positioner partition $\primarypart$, we can shift these two points next to the block $b_2$. Again by (a), we connect it to $b_2$. This yields a partition in which the blocks $b_1$ and $b_2$ are connected. Capping this partition with the pair partition erases the two auxiliary points and yields the desired partition in $\cC$.
\end{proof}

Let $p\in P(0,l)$ be a partition with $k$ blocks. We may view $p$ as a \emph{word} in $k$ letters $a_1,\ldots,a_k$ corresponding to the points connected by the partition $p$:
\[p=a_{i(1)}^{k_1}a_{i(2)}^{k_2}\ldots a_{i(n)}^{k_n} \eqstop \]
Here $a_{i(j)}\neq a_ {i(j+1)}$ for $j=1,\ldots,n-1$ and $k_j\in\N$. 
For example, the four block partition $\vierpart$ corresponds to the word $a^4$, (a rotated version of) the pair positioner partition $\primarypart$ corresponds to $ab^2ab^2$, and the double singleton $\singleton\otimes\singleton$ corresponds to $ab$. Conversely, every word $a_{i(1)}^{k_1}a_{i(2)}^{k_2}\ldots a_{i(n)}^{k_n}$ of length $l$ yields a partition $p\in P(0,l)$ connecting nothing but equal letters of the word.

For technical reasons, we introduce the \emph{empty partition} $\emptyset\in P(0,0)$ which is by definition in any category of partitions $\cC$.

\begin{lemma}\label{LemThreeRow}\label{LemReduzieren}
 Let $\cC$ be a category of partitions. Let $p\in P(0,l)$ be a partition, seen as the word $p=a_{i(1)}^{k_1}a_{i(2)}^{k_2}\ldots a_{i(n)}^{k_n}$.
\begin{itemize}
 \item[(a)]  We put $k_j':= \begin{cases} 1 &\textnormal{if $k_j$ is odd}\\ 2 &\textnormal{if $k_j$ is even}\end{cases}$, and $p':=a_{i(1)}^{k_1'}a_{i(2)}^{k_2'}\ldots a_{i(n)}^{k_n'}$.
 
 If $\cC$ contains the four block partition $\vierpart$, then $p\in\mathcal C$ if and only if $p'\in\mathcal C$. 
 \item[(b)]  We put $k_j'':= \begin{cases} 1 &\textnormal{if $k_j$ is odd}\\ 0 &\textnormal{if $k_j$ is even}\end{cases}$, and $p'':=a_{i(1)}^{k_1''}a_{i(2)}^{k_2''}\ldots a_{i(n)}^{k_n''}$. It is possible that $p''=\emptyset$.
 
 If $\cC$ contains the pair positioner partition $\primarypart$, then $p\in\mathcal C$ if and only if $p''\in\mathcal C$. 
\end{itemize}
\end{lemma}
\begin{proof}
(a) If $k_j\geq 3$, we compose $p$ with the pair partition to erase two of the neighbouring $a_{i(j)}$-points. Since this operation can be done iteratively and inside the category $\mathcal C$, we infer that $p'\in\mathcal C$ whenever $p\in\mathcal C$. For the converse, we compose $p'$ with $\idpart^{\otimes \alpha}\otimes\paarpart\otimes\idpart^{\otimes\beta}$ for suitable $\alpha, \beta$, such that the pair is situated right beside one of the $a_{i(j)}$-points of $p'$. By Lemma \ref{LemBlockVerbinden}(a), we can connect these two points to the block to which  $a_{i(j)}$ belongs, which yields a partition $\tilde p'$ where the power $k_j$ of $a_{i(j)}$ is increased by two.  By this procedure, we construct $p$ out of $p'$ inside the category $\mathcal C$.

(b) Assume first that $p''\neq \emptyset$.
If $p\in\cC$, then $p''\in \cC$ again by using the pair partition $\paarpart$. For the converse, insert pair partitions $\paarpart$ at every position in $p''$ where $k_j''=0$. By Lemma \ref{LemBlockVerbinden}(b), we can connect these pairs to the according blocks of $p''$ such that we obtain a partition $p'$ as in (a). Since the four block $\vierpart$ is in $\cC$ (see Lemma \ref{LemBezBasePart}), we conclude $p\in\cC$ using (a).

Secondly, if $p''=\emptyset$, then all exponents $k_j$ of $p$ are even. All interval partitions $q=q_1\otimes\ldots\otimes q_m$, where every partition $q_j$ consists of a single block of even length respectively, are in $\cC$. Using the fat crossing partition $\fatcrosspart$ (which is in $\cC$ by Lemma \ref{LemBezBasePart}), the partition $p$ may be obtained from a suitable interval partition $q$, by composition. Thus, $p\in\cC$.
\end{proof}

Lemma \ref{LemReduzieren}(b) will be crucial for the study of simplifiable hyperoctahedral quantum groups in the sequel. 
 
\begin{remark}
 Lemma \ref{LemReduzieren} can be extended to arbitrary partitions $p\in P(k,l)$, by rotation.
\end{remark}

In simplifiable hyperoctahedral categories, we have a notion of equivalence of partitions according to Lemma \ref{LemReduzieren}(b).

\begin{definition}
\label{def:equivalence-of-partition}
  Two partitions $p, q$ are called \emph{equivalent}, if $q$ can be obtained from $p$ by the following operations:
\begin{itemize}
 \item Elimination of two consecutive points belonging to the same block.
 \item Insertion of two consecutive points into the partition at any position and either connecting it to any other block -- or not.
\end{itemize}
\end{definition}

\begin{example}
  Both of the following partitions 
  
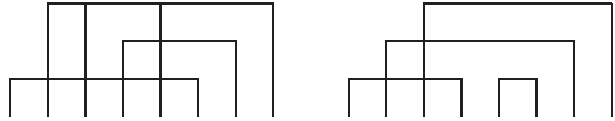
\begin{center}
  \begin{picture}(18,3)
    \put(0,0){\uppartii{1}{1}{6}}
    \put(0,0){\uppartiv{3}{2}{3}{5}{8}}
    \put(0,0){\uppartii{2}{4}{7}}
    
    \put(9,0){\uppartii{1}{1}{4}}
    \put(9,0){\uppartii{2}{2}{7}}
    \put(9,0){\uppartii{3}{3}{8}}
    \put(9,0){\uppartii{1}{5}{6}}
  \end{picture}
  \end{center}
are equivalent to the rotation of the half-liberating partition:
  
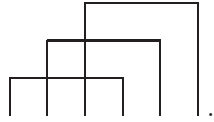
\begin{center}
  \begin{picture}(6,3)
    \put(0,0){\uppartii{1}{1}{4}}
    \put(0,0){\uppartii{2}{2}{5}}
    \put(0,0){\uppartii{3}{3}{6}}

    \put(6.5,0){.}
  \end{picture}
  \end{center}
\end{example}

\begin{lemma}
\label{lem:equivalence-stays-in-category}
  Let $\cC$ be a simplifiable hyperoctahedral category and $p \in \cC$.  Let $q$ be a partition that is equivalent to $p$.  Then $q \in \cC$.
\end{lemma}
\begin{proof}
See Lemma \ref{LemReduzieren}(b). 
\end{proof}

The category $\langle\primarypart\rangle$ is the base case for the simplifiable hyperoctahedral categories of partitions, i.e. it is contained in all simplifiable hyperoctahedral categories. We show now that the category $\langle\fatcrosspart\rangle$ is a base case for all hyperoctahedral categories that contain at least one crossing partition.

\begin{proposition}\label{PropFat}
 Let $\mathcal C$ be a hyperoctahedral category of partitions with $\mathcal C\neq\langle\vierpart\rangle$. Then the fat crossing partition $\fatcrosspart$ is in $\mathcal C$.
\end{proposition}
\begin{proof}
We show that one of the following cases hold for $\cC$:
\begin{itemize}
 \item $\fatcrosspart\in\mathcal C$.
 \item $\primarypart\in\mathcal C$.
 \item $h_s\in\mathcal C$ for some $s\geq 3$.
\end{itemize}
By Lemma \ref{LemBezBasePart} this will complete the proof.

The only hyperoctahedral category of \emph{noncrossing} partitions is $\langle\vierpart\rangle$. Thus, $\mathcal C$ contains a partition $p\in P\backslash NC$ with a crossing. We may assume that $p$ consists only of two blocks $b_1$ and $b_2$, after connecting all other blocks with one of the crossing blocks, using Lemma \ref{LemBlockVerbinden}(a).  Furthermore, we may assume that no three points in a row are connected by one of the blocks, by Lemma \ref{LemThreeRow}(a). Hence, we may write $p$ as
\[p=a^{k_1}b^{k_2}a^{k_3}b^{k_4}\ldots a^{k_n} \qquad \textnormal{or} \qquad p=a^{k_1}b^{k_2}a^{k_3}b^{k_4}\ldots b^{k_n}\]
where $k_i\in\{1,2\}$ and $n\geq 4$, and $a$ and $b$ correspond to the points connected by the blocks $b_1$ resp. $b_2$. Note that the length of $p$ is even (otherwise we could construct the singleton $\singleton$ using the pair partition).

If all $k_i=1$, then $p=h_s$ for some $s\geq 2$ -- the case $p=h_2$ implying all other cases of the claim. Otherwise, we may assume $k_1=2$ by rotation. If $n\geq 5$, we may erase the two points $a^{k_1}$ using the pair partition and we obtain a partition $p'\in\mathcal C$ which still has a crossing. Iterating this procedure, we either end up with a partition $h_s$ for some $s\geq 2$ or with a partition $p\in\mathcal C$ such that $k_1=2$ and $n=4$. In the latter case, $p$ is of length six or eight. There are exactly four cases of such a partition:
\begin{itemize}
 \item $p=aababb$ -- An application of the pair partition would yield $\singleton\otimes\singleton\in\mathcal C$ which is a contradiction.
 \item $p=aabaab$ -- This is a rotated version of $\primarypart$.
 \item $p=aabbab$ -- Again this would yield $\singleton\otimes\singleton\in\mathcal C$.
 \item $p=aabbaabb$ -- This is $\fatcrosspart$ in a rotated version.
\end{itemize}
\end{proof}

\begin{remark}
 Since $\langle\fatcrosspart\rangle$ contains a crossing partition, we have $\langle\vierpart\rangle\subsetneqq\langle\fatcrosspart\rangle$. Furthermore, we have $\langle\fatcrosspart\rangle\subset\langle\primarypart\rangle$ by Lemma \ref{LemBezBasePart}. For the proof of  $\langle\fatcrosspart\rangle\neq\langle\primarypart\rangle$ we refer to Section \ref{sectCStarAlg}.
\end{remark}

\subsection{The single leg form of a partition}
\label{sec:single-leg-form}

The pair positioner partition $\primarypart$ allows us to simplify the classification problem, since we can reduce to partitions of a nicer form.

\begin{definition}
 A partition $p\in P$ is in \emph{single leg form}, if $p$ is -- as a word -- of the form
\[p=a_{i(1)}a_{i(2)}\ldots a_{i(n)} \eqcomma\]
where $a_{i(j)}\neq a_ {i(j+1)}$ for $j=1,\ldots,n-1$.
The letters $a_1,\ldots,a_k$ correspond to the points connected by the partition $p$.
In other words, in a partition in single leg form no two consecutive points belong to the same block.

Let $\mathcal C$ be a category of partitions (or simply a set of partitions). We denote by $\mathcal C_{sl}$ the set of all partitions $p\in\mathcal C$ in single leg form. By $P_{sl}$, we denote the collection of all partitions in single leg form.
\end{definition}

\begin{lemma}\label{LemLetterTwice}
 Let $\mathcal C$ be a hyperoctahedral category and let $p\in\mathcal C_{sl}$ be a partition in single leg form. Then, every letter in the word $p$ appears at least twice. Furthermore, every word in $\cC_{sl}$ consists of at least two letters,  it has length at least four, and it is of even length.
\end{lemma}
\begin{proof}
 The double singleton $\doublesingleton$ is not contained in $\mathcal C$.
\end{proof}

If $p\in P(0,l)$ is a partition seen as the word $p=a_{i(1)}^{k_1}a_{i(2)}^{k_2}\ldots a_{i(n)}^{k_n}$, the partition $p''=a_{i(1)}^{k''_1}a_{i(2)}^{k''_2}\ldots a_{i(n)}^{k''_n}$ of Lemma \ref{LemReduzieren}(b) is not necessarily in single leg form, e.g. $p=ab^2acacaca$ yields $p''=a^2cacaca$. However, a finite iteration of the procedure as in Lemma \ref{LemReduzieren}(b) either yields a partition $q$ in single leg form or the empty partition $\emptyset\in P(0,0)$. This partition $q$ (possibly the empty partition) is called the \emph{simplified partition associated to} $p$. Note that every partition has a unique simplified partition -- the converse is not true. We can state a variation of Lemma \ref{LemReduzieren}(b).

\begin{lemma}\label{LemmaX}
 Let $\cC$ be a simplifiable hyperoctahedral category of partitions. Then, a partition is in $\cC$ if and only if its simplified partition is in $\cC$.
\end{lemma}

\begin{remark}\label{RemY}
 Every partition $p\in P$ is equivalent to its simplified partition $p''\in P$ in single leg form and two partitions are equivalent if and only if their simplified partitions agree.  If $p\in P$ is in single leg form, then the simplified partition associated to $p$ is $p$ itself. 
\end{remark}

The set $\mathcal C_{sl}$ turns out to be a complete invariant for the simplifiable hyperoctahedral categories.

\begin{proposition}\label{PropCSL}
 Let $\mathcal C$ and $\mathcal D$ be simplifiable hyperoctahedral categories.
\begin{itemize}
 \item[(a)] The category $\langle\mathcal C_{sl},\primarypart\rangle$ coincides with $\mathcal C$.
 \item[(b)] We have $\mathcal C_{sl}=\mathcal D_{sl}$ if and only if $\mathcal C=\mathcal D$.
 \item[(c)] We have $\mathcal C_{sl}\subseteq\mathcal D_{sl}$ if and only if $\mathcal C\subseteq\mathcal D$.
\end{itemize}
\end{proposition}
\begin{proof}
 (a) We have  $\langle\mathcal C_{sl},\primarypart\rangle\subset\mathcal C$. On the other hand, if $p\in \cC$, we consider its associated simplified partition $p''\in\cC$ by Lemma \ref{LemmaX}. Thus, $p''\in\langle\mathcal C_{sl},\primarypart\rangle$. Again by Lemma \ref{LemmaX}, we also have $p\in\langle\mathcal C_{sl},\primarypart\rangle$.
\end{proof}

As a consequence, we can choose the generators of a simplifiable hyperoctahedral category always to be in single leg form.
In the sequel, we will classify the subclass of simplifiable hyperoctahedral easy quantum groups by classifying the sets $\mathcal C_{sl}$.


\section{A group theoretic framework for hyperoctahedral categories of partitions}
\label{sec:group-theoretic-framework}

Denote by $L = \{a_1, a_2 ,\dotsc \}$ an infinite countable number of letters.  Let $p$ be a partition with $n$ blocks and choose a labelling $l = (a_{i(1)}, a_{i(2)}, \dotsc, a_{i(n)})$ of the blocks of $p$ with pairwise different letters.  Denote by $w(p,l)$ the word of $\freegrp{L}$ obtained by considering $p$ as a word with letters given by $l$ starting in the top left corner of $p$ and going around clockwise. Note that mutually different blocks are labelled by mutually different letters. We write $G = \ZZ_2^{* L}$ for the infinite free product of the cyclic group of order $2$ indexed by the letters in $L$.  The canonical surjection $\freegrp{L} \thra G$ is denoted by $\pi$.

The next observation describes the basic link between partitions and elements of $G$.

\begin{lemma}
\label{lem:characterization-equivalence}
 \begin{itemize}
  \item[(i)] Two partitions $p$ and $q$ are equivalent, if and only if $\pi(w(p,l)) = \pi(w(q,l'))$ for some labellings $l$ and $l'$.
  \item[(ii)] Let $\cC$ be a simplifiable hyperoctahedral category of partitions, and let $p\in\cC$ with $\pi(w(p,l))\neq e$ (where $e$ denotes the neutral element in $G$). Then there is a partition $\emptyset\neq q\in\cC$ in single leg form and a labelling $l'$ such that $\pi(w(q,l'))=\pi(w(p,l))$.
 \end{itemize}
\end{lemma}
\begin{proof}
  (i) This follows from the fact, that two words $w$ and $v$ in $\freegrp{L}$ have the same image under $\pi$ if and only if there is a sequence $w_1, \dotsc, w_n$ with $w_1=w$ and $w_n=v$ such that $w_{i +1}$ arises from $w_{i}$ by inserting or deleting a square of a letter in $L$.

  (ii) The simplified partition $p''$ associated to $p$ is in $\cC$ by Lemma \ref{LemmaX}. By (i) and Remark \ref{RemY} we get the result. Note that $p''\neq \emptyset$ since $\pi(w(p,l))\neq e$.
\end{proof}

\begin{definition}
  Let $\cC$ be a simplifiable hyperoctahedral category of partitions.  We denote by $F(\cC)$ the subset of $G$ formed by all elements $\pi(w(p,l))$ where $p \in \cC$ and $l$ runs through all possible labelling of $p$ with letters $a_1, a_2, \dotsc$.
\end{definition}

Denote by $\cP(X)$ the power set of a set $X$.  We consider the commutative diagram
\begin{center}
\begin{picture}(14,3)
  \put(0,3){$\cP(\F_L)$}
  \put(6.5,3){$\stackrel{\tilde\pi}{\longrightarrow}$}
  \put(12,3){$\cP(G)$}    

  \put(0,1.5){$w \uparrow$}
  \put(12,1.5){$w' \uparrow$}

  \put(0,0){$\cP(P)$}
  \put(6.5,0){$\stackrel{R}{\longrightarrow}$}
  \put(12,0){$\cP(P_{sl})$}
  
  \put(14.25,0){,}
\end{picture}
\end{center}
where the maps are given as follows:
\begin{itemize}
 \item The map $\tilde\pi$ is induced by the group homomorphism $\pi:\F_L\to G$, so $\tilde\pi(A):=\{\pi(x)\in G|x\in A\}$ for $A\subset\F_L$.
 \item The map $w$ is given by the above labelling of a partition with any possible choice of letters in $\F_L$, thus for a subset $A\subset P$ of partitions, we have
 \[
   w(A)
   =
   \{w(p,l)\in\F_L|p\in A, l = (a_{i(1)}, a_{i(2)}, \dotsc, a_{i(n)}) \text{ a labelling with pairwise different letters} \}
   \eqstop
 \]
 Note that we only use the generators $a_i$ of $\F_L$ as letters and not their inverses $a_i^{-1}$.
 \item The map $R$ is given by simplification of partitions. To a partition $p\in P$, we assign its simplified partition $p''\in P_{sl}$, which is possibly the empty partition. Hence
   \[
     R(A)
     =
     \{p''\in P_{sl}\;|\; p''\text{ is the simplified partition of a partition } p\in A\}
     \eqstop
   \]
 If $\cC$ is a simplifiable hyperoctahedral category, then $R(\cC)=\cC_{sl}$.
 \item The map $w'$ is given by the labelling of partitions $p$ in single leg form with any possible choice of letters in $G$, analogous to the map $w$.
\end{itemize}

We observe, that the procedure $R$ of simplifying partitions to single leg partitions corresponds to the group homomorphism $\pi$, resp. to $\tilde \pi$. Furthermore, if $\cC$ is a simplifiable hyperoctahedral category and $\pi(w(p,l))\neq e$ is an element in $\tilde\pi\circ w(\cC)$ for some partition $p\in \cC$ with some labelling $l$, we may always assume that $p$ is in single leg form. (See Lemma \ref{lem:characterization-equivalence})

We are going to study the structure of $(\tilde \pi \circ w)(\cC)$ for a simplifiable hyperoctahedral category of partitions $\cC$.  For this, we translate the category operations to operations in $\F_L$ resp. in $G$.

\begin{lemma}\label{LemFCGruppe}
 Let $\cC$ be a simplifiable hyperoctahedral category of partitions. Then:
\begin{itemize}
 \item[(i)] If the word $g=b_1\ldots b_n$ is in $w(\cC)$, then the reverse word $g'=b_n\ldots b_1$ is in $w(\cC)$. 
 \item[(ii)] If $g,h\in w(\cC)$, then $gh\in w(\cC)$.
 \item[(iii)] $(\tilde \pi \circ w)(\cC) \subset G$ is a subgroup of $G$.
\end{itemize}
\end{lemma}
\begin{proof} 
(i) Let $g=w(p,l)$ for some partition $p\in\cC$ and some labelling $l$. Thus, $w(p,l)$ is the word given by labelling the partition $p$ starting in the top left corner and going around clockwise. Since $\cC$ is a category, the partition $p^*$ is in $\cC$, given by turning $p$ upside down. Labelling $p^*$ with the letters from $l$ starting in the lower left corner and going counterclockwise yields the reverse word $g'=w(p^*,l^*)$. Here, $l^*=(a_{i(n)},  \dotsc, a_{i(1)})$ denotes the labelling of a partition in an order reverse to the one of the labelling $l$. Thus, $g'\in w(\cC)$.

(ii) Let $g=w(p,l)$ and $h=w(q,l')$, where $p,q\in \cC$, $l = (a_{i(1)}, \dotsc, a_{i(n)})$, and $l' = (a_{j(1)}, \dotsc a_{j(m)})$. By rotation, we may assume that $p$ and $q$ are partitions with no lower points. If all letters of $l$ and $l'$ are pairwise different, then $gh=w(p\otimes q,ll')$, where $ll'$ is the labelling $ll' = (a_{i(1)}, \dotsc, a_{i(n)},a_{j(1)}, \dotsc a_{j(m)})$. Otherwise, denote by $M$ the set of all pairs $(\alpha, \beta)$ in $\{1, \dotsc n\} \times \{1, \dotsc m\}$ such that $i(\alpha) = j(\beta)$. Then $gh$ is obtained from the labelled partition that is constructed by the following:
\begin{itemize}
  \item Consider the tensor product $p\otimes q$,
  \item label this partition with the letters $a_{i(1)}, \dotsc a_{i(k)}, a_{j(1)}, \dotsc ,a_{j(l)}$,
  \item now for every $(\alpha, \beta) \in M$ join the $\alpha$-th block of $p$ with the $\beta$-th block of $q$.
\end{itemize}
The resulting partition $r$ is in $\cC$  (by Lemma \ref{LemBlockVerbinden}(b)) and $gh=w(r,l'')$ with the above labelling $l''$.

(iii) If $\pi(g)\in \tilde\pi\circ w(\cC)$ for $g\in w(\cC)$, then $\pi(g)^{-1}=\pi(g')\in \tilde\pi\circ w(\cC)$ by (i). By (ii) $\tilde\pi\circ w(\cC)$ is closed under taking products.
\end{proof}

\begin{definition}
\label{def:F}
We denote by $F$ the restriction of $\tilde \pi \circ w$ to the set of all simplifiable hyperoctahedral categories of partitions as a map with image in the subgroups of $G$.
\end{definition}

This map $F$ transfers the problem of classifying the simplifiable hyperoctahedral categories of partitions to a problem in group theory.

\subsection{The correspondence between simplifiable hyperoctahedral categories and subgroups of $\ZZ_2^{* \infty}$}
\label{sec:correspondence-G}

We will give a description of the image of $F$ in terms of subgroups of $G = \ZZ_2^{*L}$ that are invariant under certain endomorphisms. This is the content of Theorem \ref{thm:F-is-1-1}.  Let us prepare its formulation.

\begin{definition}
\label{def:endomorphisms-S_0}
 Let $S_0$ be the subsemigroup of $\End(G)$ generated by the following endomorphisms.
\begin{enumerate}
\item Finite identifications of letters, i.e. for any $n \in \NN$ and any choice of indices $i(1), \dotsc, i(n)$ the map
  \[\begin{cases}
     a_k \mapsto a_{i(k)}       & 1 \leq k \leq n \eqcomma \\
     a_k \mapsto a_k           & k > n \eqstop
     \end{cases}
  \]
\item Conjugation by any letter $a_k$, i.e. the map $w \mapsto a_k \cdot w \cdot a_k$.
\end{enumerate}
\end{definition}

\begin{definition}
  Denote by $E$ the subgroup of $G$ consisting of all words of even length.
\end{definition}

Proper $S_0$-invariant subgroups of $G$ and $E$ are described in the following lemma.

\begin{lemma}
\label{lem:proper-subgroups}
 $E$ is the unique maximal proper $S_0$-invariant subgroup of $G$. Furthermore, every proper $S_0$-invariant subgroup of $E$ contains only words in which every letter $a_1, a_2, \dotsc$ appears not at all or at least twice. (Note that in $E$ itself, there are words where a letter appears only once.)
\end{lemma}
\begin{proof}
   Firstly, note that $E$ has index $2$ in $G$, so it is a maximal proper subgroup of $G$. Secondly, it is $S_0$-invariant. Now, if an $S_0$-invariant subgroup $H \leq G$ contains a word with an odd number of letters, say $2n + 1$, we may use the identification of letters from Definition \ref{def:endomorphisms-S_0}(i) in order to obtain $a_1 = a_1^{2n + 1} \in H$.  With $a_1 \in H$, it follows that $a_i \in H$ for all $i$ and hence $H = G$.

Let $H \leq E$ be an $S_0$-invariant subgroup and assume that there exists an element $w \in H$ where $w$ contains a letter $a_i$ only once.  Using identification of letters, we may assume that $i = 1$ and all other letters are the same, say $a_2$.  We obtain $a_1a_2 \in H$ or $a_2a_1 \in H$, thus $a_ia_j\in H$ for all $i,j$.  Now let $w \in E$ be arbitrary.  We can write
\[w = a_{i(1)}a_{i(2)} \dotsm a_{i(2n)} = (a_{i(1)}a_{i(2)})(a_{i(3)}a_{i(4)}) \dotsm (a_{i(2n-1)}a_{i(2n)}) \in H \eqcomma\]
for some indices $i(1), i(2), \dotsc, i(2n)$.  So $H = E$ and we have finished the proof.
\end{proof}

\begin{lemma}
\label{lem:F-defines-invariant-group}  
  For any simplifiable hyperoctahedral category of partitions $\cC$, $F(\cC)$ is a proper $S_0$-invariant subgroup of $E$.  

So $F$ is a well-defined map from simplifiable hyperoctahedral categories of partitions to proper $S_0$-invariant subgroups of $E$.  Moreover, $F$ is a lattice homomorphism.
\end{lemma}
\begin{proof}
By Lemma \ref{LemFCGruppe} $F(\cC)$ is a subgroup of $G$.  Since all partitions in $\cC$ are of even length, $F(\cC)$ is a subgroup of $E$.  

Assume that $F(\cC)=E$.  Then $F(\cC)$ contains an element in which some letter appears only once.  Hence, in $\cC$ there is a partition with a singleton. Thus $\cC$ is not hyperoctahedral, which is a contradiction.  We have shown that $F(\cC)$ is a proper subgroup of $E$.

We show that $F(\cC)$ is invariant under the generating endomorphisms of $S_0$ in Definition \ref{def:endomorphisms-S_0}.  Let $g = \pi(w(p,l))$ be an element in $F(\cC)$ constructed from a partition $p \in \cC$.  It is clear that we can change a letter in $g$, if the new letter did not appear in $g$ before -- this simply corresponds to $\pi(w(p,l'))$ with a different labelling $l'$.  If the new letter already appeared in $g$, we connect two blocks of $p$ using Lemma \ref{LemBlockVerbinden}.  This shows that $F(\cC)$ is closed under identification of letters. 

Furthermore, $F(\cC)$ is closed under conjugation with a letter $a_k$.  Indeed, let $e \neq g = \pi(w(p,l)) = a_{i(1)} \dotsc a_{i(m)}$ be an element in $F(\cC)$.  Assume that $p$ is a partition in single leg form with no lower points (see Lemma \ref{lem:characterization-equivalence}).
 If the letter $a_k$ does not appear in the word $a_{i(1)} \dotsc a_{i(m)}$, we consider the partition
\begin{center}
\begin{picture}(4,3)
  \put(0,0.75){$p' =$}
  \put(1,0){$\uppartii{2}{1}{4}$}
  \put(3.25,0.5){$p$}

  \put(5.25,0.25){,}
\end{picture}
\end{center}
i.e. the partition obtained from $p$ by nesting it into a pair partition $\paarpart$. Labelling this partition with $l'=(a_k,a_{j(1)},\ldots, a_{j(n)})$ for $l=(a_{j(1)},\ldots, a_{j(n)})$ yields $a_k g a_k = \pi(w(p',l'))$ in $F(\cC)$. On the other hand, if the letter $a_k$ appears in the word $a_{i(1)} \dotsc a_{i(m)}$, we have four cases.
\begin{itemize}
 \item If $i(1) \neq k$ and $i(m)\neq k$, we connect the outer pair partition of $p'$ with the block of $p$ which corresponds to the letter $a_k$ (see Lemma \ref{LemBlockVerbinden}).  The resulting partition $p''$ yields $a_k g a_k = \pi(w(p'',l''))$ in $F(\cC)$ for a suitable labelling $l''$.
 \item If $i(1) \neq k$ and $i(m) = k$, the element $a_k g a_k$ is given by $a_k g a_k = a_{i(m)}a_{i(1)} \dotsc a_{i(m-1)}$ (as $a_k^2=e$).  Therefore, we consider the labelled partition $p$ in a rotated version, which yields $a_k g a_k \in F(\cC)$. Likewise in the case $i(1)= k$ and $i(m) \neq k$.
 \item If $i(1) = k$ and $i(m) = k$, the element $a_k g a_k$ equals $a_{i(2)} \dotsc a_{i(m-1)}$. On the other hand, the very left point and the very right point of $p$ belong to the same block, so, rotating one of them next to the other and erasing them using the pair partition yields a partition $p'' \in \cC$ such that $a_k g a_k = \pi(w(p'',l''))$ for a suitable labelling $l''$.
\end{itemize}
Finally note that, since $\tilde \pi \circ w$  preserves inclusions, its restriction $F$ is a lattice homomorphism.
\end{proof}

The preceding lemma specifies that we can associate an $S_0$-invariant subgroup $F(\cC)$ of $E$ to any simplifiable hyperoctahedral category $\cC$ of partitions -- but we can also go back. In fact, \emph{every} proper $S_0$-invariant subgroup of $E$ comes from such a category. 
This is worked out in the sequel.

\begin{lemma}
For any proper $S_0$-invariant subgroup $H$ of $E$, the set 
\[\cC_H:=w^{-1} ( \pi^{-1}(H))=\{p\in P \amid \textnormal{there is a labelling $l$ such that } \pi(w(p,l))\in H\}\subset P\]
 is a simplifiable hyperoctahedral category of partitions.
\end{lemma}
\begin{proof}
The pair partition $\sqcap$, the unit partition $\idpart$, the four block partition $\vierpart$, and the pair positioner partition $\primarypart$ are all in $\cC_H$, since they are mapped to the neutral element $e\in H$ for any labelling $l$. 

Let $p$ and $q$ be partitions in $\cC_H$ and denote by $g:=\pi(w(p,l))$ and $h:=\pi(w(q,l'))$ some corresponding elements in $H$ for some labellings $l$ and $l'$. Since $H$ is invariant under permutation of letters we can assume that the labellings $l$ and $l'$ are such that $g$ and $h$ do not share any letter. The element $g$ may be written as $g=g_1g_2$, where $g_1$ corresponds to the labelling of the upper points of $p$, and $g_2$ to the lower points of $p$.
Consider the tensor product $p\otimes q$ of $p$ and $q$ as labelled partitions, i.e. we form $p\otimes q$ and label it by a labelling $l''$ such that the subpartition $p$ in $p\otimes q$ is labelled by $l$ and the subpartition $q$ is labelled by $l'$. Then, the element $\pi(w(p\otimes q,l''))$ is of the form $g_1hg_2$. (Recall that the labelling procedure starts at the upper left point of a partition and goes around clockwise -- thus, in $p\otimes q$ the upper points of $p$ are labelled first, then the whole of $q$ is labelled, and we finish by labelling the lower points of $p$.) As $H$ is closed under conjugation, the element $g_1hg_1^{-1}$ is in $H$, so is $g_1hg_2=g_1hg_1^{-1}g$. Hence, $p\otimes q\in \cC_H$, and $\cC_H$ is closed under tensor products.

The set $\cC_H$ is also closed under involution, since for $p\in \cC_H$ with $\pi(w(p,l))=g\in H$, we have $\pi(w(p^*,l^*))=g^{-1}\in H$, where $l^*$ denotes the labelling $l$ in reverse order.  It is also closed under rotation, since moving points (from above to below or the converse) at the right hand side of a partition $p$ does not change the labelling -- and hence $\pi(w(p,l))$ is invariant under this operation. Moving points at the left hand side of $p$ is reflected by conjugating $\pi(w(p,l))$ with the first respectively with the last letter of this word.

It remains to show that $\cC_H$ is closed under the composition of partitions. We first show that $\cC_H$ is closed under composition with a partition of the form

\begin{center}
\setlength{\unitlength}{0.5cm}
\begin{picture}(10,2)
  \put(0,0){\line(0,1){2}}
  \put(0.8,0.8){$\dotsm$}
  \put(2,0){\line(0,1){2}}
  \put(3,1){\line(0,1){1}}
  \put(4,1){\line(0,1){1}}
  \put(3,1){\line(1,0){1}}
  \put(5,0){\line(0,1){2}}
  \put(5.8,0.8){$\dotsm$}
  \put(7,0){\line(0,1){2}}

  \put(7.25,0){.}
\end{picture}
\end{center}

Let $p\in\cC_H$ be a partition on $k$ upper points and $m$ lower points and consider the partition $|| \dotsm | \raisebox{0.4ex}[0ex][0ex]{$\sqcup$} || \dotsm |$ on $m$ upper points and $m-2$ lower points, where $\sqcup$ connects the $i$-th and the $(i + 1)$-st point. Denote their composition by $p'$. There is a labelling $l$ such that $g:=\pi(w(p,l))$ is in $H$.  For a suitable labelling $l'$, the element $\pi(w(p',l'))$ arises from $g$ by identifying the $(k+i)$-th and the $(k+i+1)$-st letter.  Since $H$ is invariant under this operation, the partition $p'$ is in $\cC_H$.

It remains to reduce the composition of arbitrary partitions to the previous case. Let $p\in\cC_H$ be a partition on $k$ upper and $l$ lower points, and let $q\in \cC_H$ be on $l$ upper and $m$ lower points. Write $p'$ and $q'$ for the partitions arising from $p$ and $q$, respectively, by rotating their lower points to the right of the upper points. Then $p'$ and $q'$ are both in $\cC_H$.  Composing $p'\otimes q'$ with the partition

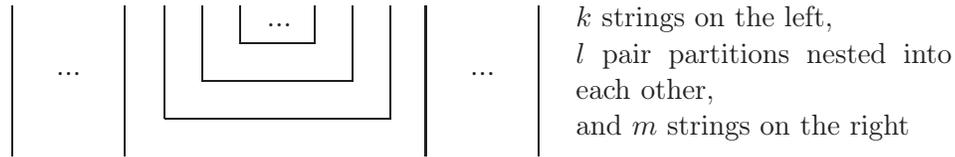
\begin{center}
\setlength{\unitlength}{0.5cm}
\begin{picture}(24,4)
  \put(0,0){\line(0,1){4}}
  \put(1.2,2){$\dotsm$}
  \put(3,0){\line(0,1){4}}

  \put(2.8,5){\partii{3}{1}{7}}
  \put(2.8,5){\partii{2}{2}{6}}
  \put(2.8,5){\partii{1}{3}{5}}

  \put(6.8,3.3){$\dotsm$}

  \put(11,0){\line(0,1){4}}
  \put(12.2,2){$\dotsm$}
  \put(14,0){\line(0,1){4}}

  \put(15,2){\begin{minipage}{5cm} $k$ strings on the left, \\
                                   $l$ pair partitions nested into each other, \\
                                   and $m$ strings on the right \end{minipage}}
\end{picture}
\end{center}

yields a partition $p''\in\cC_H$ on $k+m$ points. Rotating $m$ points on the right of $p''$ to below gives the composition $pq$ of $p$ and $q$, which hence is in $\cC_H$.

We conclude that $\cC_H$ is closed under the category operations, hence it is a category of partitions, containing $\primarypart$. On the other hand, the partition $\singleton\otimes\singleton$ is not in $\cC_H$, since $\pi(w(\singleton\otimes\singleton,l))$ is a word of the form $ab$, where $a$ and $b$ are different letters in $G$. By Lemma \ref{lem:proper-subgroups}, these elements are not in $H$. Thus, $\cC_H$ is hyperoctahedral and simplifiable.
\end{proof}

We show now that the map $H \mapsto \cC_H$ is the inverse of $F$.

\begin{theorem}
\label{thm:F-is-1-1}
The maps $F$ and $H \mapsto \cC_H$ are inverse to each other. Hence, the map $F$ is bijective as a map from simplifiable hyperoctahedral categories of partitions onto proper $S_0$-invariant subgroups of $E$. 
\end{theorem}
\begin{proof}
Firstly, let $H$ be a proper $S_0$-invariant subgroup of $E$, and let $x \in H$.  Denote by $p$ the partition connecting the letters of the word $x$ if and only if they coincide, and let $l$ be the labelling such that $\pi(w(p,l)) = x$.  Thus, $p \in \cC_H$ and hence $x \in F(\cC_H)$. (Recall that $x\in F(\cC)$ if and only if $x=\pi(w(p,l))$ for some $p\in \cC$ and some labelling $l$.)  Conversely, let $x = \pi(w(p,l)) \in F(\cC_H)$ where $p \in \cC_H$. By definition, there is a labelling $l'$ such that $\pi(w(p,l'))\in H$. Now, $H$ is invariant under exchange of letters, thus $x=\pi(w(p,l))\in H$. We deduce that $H = F(\cC_H)$.

Secondly, let $\cC$ be a simplifiable hyperoctahedral category of partitions, and let $p \in \cC$. Then $\pi(w(p,l)) \in F(\cC)$ for any labelling $l$, and hence $p \in \cC_{F(\cC)}$. On the other hand, for $p \in \cC_{F(\cC)}$ there is a labelling $l$ such that $\pi(w(p,l)) \in F(\cC)$. Thus, $\pi(w(p,l)) = \pi(w(q,l'))$ for some partition $q \in \cC$ and some labelling $l'$.  By Lemma \ref{lem:characterization-equivalence} and Lemma \ref{lem:equivalence-stays-in-category}, we have $p \in \cC$.  This finishes the proof of $\cC = \cC_{F(\cC)}$.
\end{proof}


\section{Classification and structural results for easy quantum groups}
\label{sec:classification-and-structure}

In this section we deduce from Theorem \ref{thm:F-is-1-1} that there are uncountably many different hyperoctahedral categories.
We end this section by giving structural results on the lattice of hyperoctahedral quantum groups.

\subsection{The classification of simplifiable hyperoctahedral categories by invariant subgroups of $\freegrp{\infty}$}
\label{sec:correspondence-finfty}

We can identify $E \leq \ZZ_2^{* \infty}$ with a free group and describe the restriction of endomorphisms from $S_0$ to $E$.  This is the content of the next lemma.

\begin{lemma}
\label{lem:E-is-free+restriction-of-S_0}
We put $x_k:=a_1a_{k+1}$ for $k=1,2,\ldots$ Then $x_1, x_2, \dotsc$ is a free basis of $E$. The restriction $\{\phi_{|E}|\phi\in S_0\subset \End(G)\}$ of endomorphisms to $E$ is the semigroup generated by
\begin{enumerate}
\item finite identifications of letters, i.e. for any $n \in \NN$ and any choice of indices $i(1), \dotsc, i(n)$ the map
  \[\begin{cases}
     x_k \mapsto x_{i(k)}       & 1 \leq k \leq n \eqcomma \\
     x_k \mapsto x_k           & k > n \eqcomma
     \end{cases}
  \]
\item for all $i \in \NN$, the map defined by $x_k \mapsto x_i^{-1} \cdot x_k$ for all $k \in \NN$,
\item for all $k \in \NN$, the map $x_k \mapsto e$ that leaves all other letters invariant,
\item the map $x_k \mapsto x_k^{-1}$ for all $k \in \NN$,
\item all inner automorphisms of $E$, 
\end{enumerate}
\end{lemma}
\begin{proof}
  We have $x_k^{-1} = a_{k + 1}a_1$ for all $k \in \NN$.  So, every word of even length in $G$ can be written uniquely as a product of the elements $x_1, x_1^{-1}, x_2, x_2^{-1}, \dotsc$, so $x_1, x_2, \dotsc$ is a free basis for $E$.

We check that the endomorphisms (i)-(iv) are precisely the restrictions of the generators of $S_0$ as given in Definition \ref{def:endomorphisms-S_0}.  In order to obtain (i) it suffices to consider the endomorphism of $G$ defined by $a_{k + 1} \mapsto a_{i(k) + 1}$ for $1 \leq k \leq n$  and leaving all other letters invariant.  For (ii), we have to consider the endomorphism of $G$ mapping $a_1 \mapsto a_{i + 1}$ and leaving all other letters invariant.  For (iii), we have to take the map $a_{k+1} \mapsto a_1$.  We considered all possible endomorphisms from item (i) in Definition \ref{def:endomorphisms-S_0}.  The endomorphisms in (iv) are obtained by mapping $a_{k + 1} \mapsto a_1 a_{k + 1} a_1$ and for all $k \in \NN$.  The conjugation by $x_i$ is obtained by $a_k \mapsto a_{i + 1} a_k a_{i + 1}$ for all $k \in \NN$ composed with the endomorphism in (iv).  Indeed $x_k = a_1a_{k + 1}$ is mapped to $a_{i + 1} a_1 a_{k+1} a_{i + 1} = x_i^{-1} x_k^{-1} x_i$ and the endomorphism in (iv) maps this element to $x_i x_k x_i^{-1}$.  We considered all possible endomorphisms from both items (i) and (ii) in Definition \ref{def:endomorphisms-S_0}, so the restriction of $S_0$ to $E$ is equal to the semigroup described in the statement.
\end{proof}

Note that the endomorphisms defined in the previous lemma depend on the choice of the free basis $x_1,x_2, \dotsc$. Fixing this choice, we obtain an isomorphism $E \cong \freegrp{\infty}$.

\begin{definition}
\label{def:S}
  We denote by $S$ the subsemigroup of $\End(\freegrp{\infty})$ generated by the maps in Lemma \ref{lem:E-is-free+restriction-of-S_0}.
\end{definition}

\begin{lemma}
\label{lem:S-invariant-subgroups}
  A subgroup of $\freegrp{\infty}$ is $S$-invariant, if and only if
  \begin{enumerate}
  \item it is closed under identification and deletion of letters,
  \item for all $i \in \NN$ it is closed under the map $x_k \mapsto x_i^{-1} \cdot x_k$ for all $k \in \NN$,
  \item it is closed under the map $x_k \mapsto x_k^{-1}$ for all $k\in\NN$,
  \item and it is normal.
  \end{enumerate}
\end{lemma}
\begin{proof}
  This is a translation of Lemma \ref{lem:E-is-free+restriction-of-S_0}.  Item (i) and (iii) of Lemma \ref{lem:E-is-free+restriction-of-S_0} correspond to item (i) here. The items (ii) correspond to each other, item (iv) of Lemma \ref{lem:E-is-free+restriction-of-S_0} corresponds to item (iii) here and normality of an $S$-invariant subgroup is the same as invariance under inner automorphism (item (v) of Lemma \ref{lem:E-is-free+restriction-of-S_0}). 
\end{proof}

Theorem \ref{thm:F-is-1-1} can be now translated into the more convenient setting of subgroups of $\freegrp{\infty}$.

\begin{theorem}
\label{thm:F-is-isomorphism}
  The map $F$ of Definition \ref{def:F} induces a lattice isomorphism between simplifiable hyperoctahedral categories of partitions and proper $S$-invariant subgroups of $\freegrp{\infty}$.
\end{theorem}
\begin{proof}
  This follows from Theorem \ref{thm:F-is-1-1} and Lemma \ref{lem:E-is-free+restriction-of-S_0}
\end{proof}

The formulation of the previous theorem allows us to employ a well-known subset of the $S$-invariant subgroups of $\freegrp{\infty}$.  The next observation is essential for the rest of this section.

\begin{remark}
\label{rem:fully-characteristic-is-S-invariant}
Every fully characteristic subgroup (see Section \ref{sec:introduction-to-varieties} for a definition) of $\freegrp{\infty}$ is $S$-invariant.  Hence $F$ induces a lattice embedding of proper fully characteristic subgroups of $\freegrp{\infty}$ into  simplifiable hyperoctahedral categories of partitions.  
\end{remark}

To close this section, let us ask whether or not also the other implication holds:  
Is every $S$-invariant subgroup of $\freegrp{\infty}$ fully characteristic?
We only have a partial answer to this question.

\begin{proposition}
\label{prop:conjecture-for-abelian-varieties}
  Every $S$-invariant subgroup of $\freegrp{\infty}$ that contains the commutator $x_1x_2x_1^{-1}x_2^{-1}$ is fully characteristic.
\end{proposition}
\begin{proof}
It suffices to show that the map sending an $S$-invariant subgroup of $\freegrp{\infty}$ to its fully characteristic closure is injective on subgroups containing $x_1x_2x_1^{-1}x_2^{-1}$.  A proper $S$-invariant subgroup $H \leq \freegrp{\infty}$ contains the commutator $x_1x_2^{-1}x_1^{-1}x_2 = a_1a_2a_3a_1a_2a_3$ if and only if the associated category of partitions contains the half-liberating partition.  Similarly, the $s$-mixing partition $h_s$ corresponds to the element $a_1a_2a_1a_2 \dotsc a_1a_2$ ($s$ repetitions), which is equal to $x_1^s$.  So by \cite[Theorem 4.13]{weber12}, it suffices to prove that the fully characteristic subgroups generated by $x_1x_2x_1^{-1}x_2^{-1}$ and $x_1^s$ are pairwise different for different $s \in \NN \setminus \{1\}$.  By the fact that the group $\ZZ/s\ZZ$ is abelian and has exponent $s$ but not exponent $s'$ for $s' < s$, invoking the correspondence between fully characteristic subgroups of $\freegrp{\infty}$ and varieties of groups from Theorem \ref{thm:varieties-and-fully-characteristic-subgroups}, we finish the proof.
\end{proof}

\subsection{Classification results for easy quantum groups}
\label{sec:classification-results}

The link between the theory of varieties of groups and easy quantum groups is given by Theorem \ref{thm:F-is-isomorphism} and Remark \ref{rem:fully-characteristic-is-S-invariant}.  Let us state this more precisely. A hyperoctahedral easy quantum group is called \emph{simplifiable}, if its corresponding hyperoctahedral category of partitions is simplifiable.

\begin{theorem}
\label{thm:easy-quantum-groups-and-varieties}
  There is a lattice injection from the lattice of nonempty varieties of groups into the lattice of simplifiable hyperoctahedral easy quantum groups.
\end{theorem}
\begin{proof}
  The lattice of simplifiable hyperoctahedral quantum groups is anti-isomorphic to the lattice of simplifiable hyperoctahedral categories of partitions. Theorem \ref{thm:F-is-isomorphism} shows that the latter lattice is isomorphic to the lattice of proper $S$-invariant subgroups of $\freegrp{\infty}$.  By Remark \ref{rem:fully-characteristic-is-S-invariant}, there is an injection of lattices of proper fully characteristic subgroups of $\freegrp{\infty}$ into the lattice of proper $S$-invariant subgroups of $\freegrp{\infty}$.  The former is anti-isomorphic with the lattice of nonempty varieties of groups by Theorem \ref{thm:varieties-and-fully-characteristic-subgroups}.  Composing these isomorphisms, injections and anti-isomorphisms, we obtain an injection of lattices as in the statement of the theorem.  
\end{proof}

\begin{remark}
\label{rem:varieties-combinatorics}
  The proof of the previous theorem also shows that there is a one-to-one correspondence between varieties of groups and certain simplifiable hyperoctahedral categories of partitions.  We hence obtain a combinatorial and a quantum group perspective on varieties of groups.
\end{remark}

The correspondence from the last theorem allows us to translate known results about varieties of groups into statements about easy quantum groups. Let us start with some results about the classification of easy quantum groups.

In \cite{banicacurranspeicher09_2}, the question was raised whether or not all easy quantum groups are either classical, free, half-liberated or form part of a multi-parameter family unifying the series of quantum groups $H_n^{(s)}$ and $H_n^{[s]}$.  We can answer this question in the negative.

\begin{theorem}
\label{thm:uncountably-many-easy-quantum-groups}
  There are uncountably many pairwise non-isomorphic easy quantum groups.
\end{theorem}

This follows directly from Theorem \ref{thm:easy-quantum-groups-and-varieties} and the following result by Olshanskii.

\begin{theorem}[See \cite{olshanskii70}]
\label{thm:uncountably-many-varieties}
  The class of varieties of groups has cardinality equal to the continuum.
\end{theorem}

Easy quantum groups offer a class of examples, which is concretely accessible by means of combinatorics.  Therefore, it would be good to amend Theorem \ref{thm:uncountably-many-easy-quantum-groups} with concrete examples.  Unfortunately, partitions are not well-suited to represent higher commutators in $E$.  We therefore omit a concrete translation of the following result of Vaughan-Lee.  The notation $[x_1, x_2, x_3, \dotsc, x_n]$ denotes the higher commutator $[[[ \dotsm [[x_1,x_2], x_3],x_4], \dotsc ],x_n]$.

\begin{theorem}[See \cite{vaughan-lee70}]
 Let $x,y,z, x_1,x_2, \dotsc$ be a free basis of $\freegrp{\infty}$. Denote $w_k = [[x,y,z],[x_1,x_2],[x_3,x_4], \dotsc , [x_{2k-1},x_{2k}],[x,y,z]]$. Then the fully characteristic subgroups of $\freegrp{\infty}$ generated by
\[\{w_k \amid k \in I\} \cup \{x^{16}, [[x_1,x_2,x_3],[x_4,x_5,x_6],[x_7,x_8]]\} \eqcomma \quad I \subset \NN\]
are pairwise different.
\end{theorem}

It would be interesting to find an uncountable family of categories of partitions, which is more natural from the point of view of combinatorics.

\begin{remark}
\label{rem:complexity-of-easy-quantum-groups}
Given a certain class of objects that we want to classify up to a certain equivalence relation $\simeq$, it is generally known that the cardinality of the quotient set does not say a lot about the difficulty of the classification. The difference between the classification of torsion free abelian groups of rank $1$ \cite{bear37} and of rank $\geq 2$ \cite{kurosch37, malcev38} is a classical instance of this fact.  The theory of \emph{Borel reducibility} offers a better point of view on classification problems of this kind.  See \cite{kechris99-new-directions} for an exposition.  If $\cR, \cS$ are equivalence relations on  Polish spaces $X$ and $Y$, respectively, then $\cR$ is called Borel reducible to $\cS$, if there is a Borel map $f: X \ra Y$ such that $x_1 \simeq_\cR x_2 \Leftrightarrow f(x_1) \simeq_\cS f(x_2)$.  In common terms, $\cR$ is ``easier'' then $\cS$.  We call $\cR$ \emph{smooth}, if it is reducible to the equivalence relation of equality of points on some Polish space $Y$.

Denoting by $P$ the set of all partitions, the space $X = 2^P$ of subsets of $P$ is a Polish space.  Denote by $\cR_{QG}$ the equivalence relation on $X$ making $x_1, x_2 \in X$ equivalent, if and only if they generate the same category of partitions, i.e. $\langle x_1 \rangle = \langle x_2 \rangle$.  We want to pose the question, whether the equivalence relation $\cR_{QG}$ is smooth.

Note that by \cite[Theorem 1.1]{harringtonkechrislouveau90}, in order to show that $\cR_{QG}$ is not smooth, it suffices to find a family $(p_n)_{n \in \NN}$ of partitions, such that two subsets $\{ p_n \amid n \in A\}$ and $\{ p_n \amid n \in B\}$ generate the same category of partitions if and only if $A \Delta B$ is finite. 
\end{remark}

\begin{remark}
\label{rem:tensor-categories-of-easy-quantum-groups}
If $(A, u)$ is a hyperoctahedral quantum group with associated category of partitions $\cC$, then $\singleton \ot \singleton \notin \cC$.  As a consequence, $\Hom(u,u)$ is one dimensional and hence $u$ is an irreducible corepresentation of $A$.  It follows that the tensor \Cstar-category of unitary finite dimensional corepresentations of $(A, u)$ is generated by a single irreducible element.  So Theorem \ref{thm:uncountably-many-easy-quantum-groups} gives rise to many new tensor \Cstar-categories, whose fusion rules are described by the combinatorics of categories of partitions.  It remains an interesting question to determine these fusion rules.
\end{remark}

\subsection{Structural results for easy quantum groups}
\label{sec:structural-results}

We position known hyperoctahedral quantum groups in the context of varieties of groups.  See also Example \ref{ex:varieties}.

\begin{example}
\label{ex: varieties-and-easy-quantum groups}
  \begin{enumerate}
  \item The variety of all groups corresponds to the trivial subgroup of $\freegrp{\infty}$, which in turn corresponds to the maximal simplifiable quantum group $H_n^{[\infty]}$ (resp. to the category $\langle\primarypart\rangle$).
  \item By Proposition \ref{prop:conjecture-for-abelian-varieties}, the category $\langle \halflibpart,\vierpart \rangle$ corresponds to the commutator subgroup of $\freegrp{\infty}$.  So Example \ref{ex:varieties}(ii) shows that the quantum group $H_n^*$ corresponds to the variety of all abelian groups.
  \item By the same Proposition and Examples \ref{ex:varieties} (ii) and (iii), the categories $\langle \halflibpart, \vierpart, h_s \rangle$ correspond to the fully characteristic subgroup of $\freegrp{\infty}$ generated by the commutator subgroup and $x_1^s$.  It follows that the easy quantum group $H_n^{(s)}$, $s \geq 2$ corresponds to the variety of abelian groups of exponent $s$.  Note that $H_n^{(2)} = H_n$ is a group.
  \end{enumerate}
\end{example}

We end this section by giving two structural results regarding the classification of hyperoctahedral quantum groups.

For a simplifiable hyperoctahedral category $\cC$ denote by $\cC^0$ its intersection with $\langle \halflibpart, \vierpart \rangle$.  The following theorem generalises Theorem 4.13 of \cite{weber12}.

\begin{proposition}
\label{prop:reduction-of-classification}
  Every simplifiable hyperoctahedral easy quantum group $G$ is either an intermediate quantum subgroup $H_n^+ \supset G \supset H_n^*$ or it is the intersection of some $H_n^+ \supset G_0 \supset H_n^*$ and  $H_n^{[s]}$ for some $s \geq 2$.  Equivalently, for every simplifiable hyperoctahedral category $\cC$ that is not contained in $\langle \halflibpart, \vierpart \rangle$, there is $s \geq 2$ such that $\cC = \langle \cC^0, h_s \rangle$.
\end{proposition}
\begin{proof}
  Take a simplifiable hyperoctahedral category $\cC$ that does not contain the crossing partition.  Let $H \leq \freegrp{\infty}$ be the $S$-invariant subgroup of $\freegrp{\infty}$ associated with $\cC$ by Theorem \ref{thm:F-is-isomorphism}.

Take $w \in H$.  The exponent of $x_i$ in $w$ is by definition the sum of the powers of $x_i$ that appear in $w$.  Denote by $e_i$ the exponent of $x_i$.  For all $i$, we obtain $x_i^{e_i} \in H$, by applying to $w$ the endomorphism of $\freegrp{\infty}$ that erases all letters of $w$ except for $x_i$.  For later use, note that since $[\freegrp{\infty}, \freegrp{\infty}]$ is the kernel of the abelianisation map $\freegrp{\infty} \ra \ZZ^{\infty}$, we can write for some $n \in \NN$ and for some word $c \in H \cap [\freegrp{\infty}, \freegrp{\infty}]$
\[
  w
  =
  x_1^{e_1} \dotsm x_n^{e_n} (x_n^{-e_n} \dotsm x_1^{-e_1} w)
  =
  x_1^{e_1} \dotsm x_n^{e_n}c
  \eqstop
\]
Let $s$ be the minimal number such that $x_1^s \in H$.  By the previous decomposition of words we see that $H$ is generated as an $S$-invariant subgroup by $x_1^s$ and by $H \cap [\freegrp{\infty}, \freegrp{\infty}]$.  Appealing to the correspondence between simplifiable hyperoctahedral categories of partitions and $S$-invariant subgroups of $\freegrp{\infty}$ in Theorem \ref{thm:F-is-isomorphism}, we have finished the proof.
\end{proof}

\begin{proposition}
\label{prop:maximal-hyperoctahedral-quantum-groups}
  Every simplifiable hyperoctahedral quantum group $G\neq H_n$ has $H_n^{(s)}$ as a quantum subgroup for some $s \geq 3$.  Equivalently, every simplifiable hyperoctahedral category of partitions that does not contain the crossing $\crosspart$ is contained in $\langle \halflibpart, \vierpart, h_s \rangle$ for some $s \geq 3$.
\end{proposition}
\begin{proof}
  This follows from Proposition \ref{prop:reduction-of-classification}: let $G \neq H_n$ be a simplifiable hyperoctahedral quantum group.  Then either $G$ contains $H_n^*$ or it contains $H_n^{(s)} = H_n^* \cap H_n^{[s]}$ for some $s \geq 3$.
\end{proof}


\section{The \Cstar-algebras associated to the simplifiable hyperoctahedral categories}
\label{sectCStarAlg}

Given a category of partitions $\cC$, we denote by $(A_\cC(n), u_n)$ the compact matrix quantum group with fundamental corepresentation of size $n \times n$ associated with $\cC$.  In this section, we study the \Cstar-algebras associated with the categories $\cC = \Katprim$, denoted by $\AKat = \rC(H_n^{[\infty]})$.  We also study the $\Cstar$-algebras $\AKat$ for $\mathcal C=\Katfat$, since some of their theory is similar.

Recall that the hyperoctahedral quantum group $H_n^+$ corresponds to the category $\langle\vierpart\rangle$. If $G\subset O_n^+$ is a compact quantum subgroup of $O_n^+$ and $u$ denotes the fundamental corepresentation of $\rC(G)$, then the map $T_p$ for $p=\vierpart$ is in the intertwiner space $\Hom(1, u^{\ot 4})$ if and only if $u_{ik}u_{jk}=u_{ki}u_{kj}=0$ whenever $i\neq j$. Hence, this relation is fulfilled for all compact quantum subgroups $G\subset H_n^+$. We also have $\sum_k u_{ik}^2=\sum_ku_{kj}^2=1$, as well as $u_{ij}=u_{ij}^*$, for all $i,j$ (see \cite{weber12} for the relations of $C(H_n^+)$).

Recall also the definition of the linear maps $T_p:(\C^n)^{\otimes k}\to(\C^n)^{\otimes l}$ indexed by a partition  $p\in P(k,l)$:
\[
  T_p(e_{i(1)}\otimes \dotsm \otimes e_{i(k)})
  =
  \sum_{j(1),\ldots,j(l)=1}^n \delta_p(i,j) \cdot e_{j(1)}\otimes \dotsm \otimes e_{j(l)}
  \eqstop
\]
Here, $e_1,\ldots,e_n$ is the canonical basis of $\C^n$, and $\delta_p(i,j)=1$ if and only if the indices of $i=(i(1), \dotsc, i(k))$ and $j=(j(1), \dotsc, j(l))$ that are connected by the partition $p$ coincide.  Otherwise $\delta_p(i,j)=0$.
(cf. \cite[Definitions 1.6 and 1.7]{banicaspeicher09})

\begin{lemma}
\label{lem:relations-and-intertwiners}
 Let $G\subset H_n^+$ be a compact quantum subgroup of $H_n^+$ and denote by $\rC(G)$ its corresponding \Cstar-algebra generated by the entries of the fundamental corepresentation $u_{ij}$, $i,j = 1, \dotsc, n$.  Then
\begin{enumerate}
 \item $T_p \in \Hom(u^{\ot 4}, u^{\ot 4})$ for $p = \fatcrosspart$ if and only if $u_{ij}^2u_{kl}^2=u_{kl}^2u_{ij}^2$ for all $i,j,k,l$.
 \item $T_p \in \Hom(u^{\ot 3}, u^{\ot 3})$ for $p = \primarypart$ if and only if $u_{ij}u_{kl}^2=u_{kl}^2u_{ij}$ for all $i,j,k,l$.
\end{enumerate}
\end{lemma}
\begin{proof} 
Compare $u^{\otimes 4}(T_p\otimes 1)$ with $(T_p\otimes 1)u^{\otimes 4}$ for (i) and analogous for (ii).
\end{proof}

We will now describe the \Cstar-algebras corresponding to the categories $\langle\fatcrosspart\rangle$ and $\langle\primarypart\rangle$. 

\begin{proposition}\label{GestaltAprimary}
The \Cstar-algebras $\AKat$ associated with $\Katfat$ and $\Katprim$ are universal \Cstar-algebras generated by elements $u_{ij}$, $i,j,=1,\ldots,n$ such that
\begin{enumerate}
 \item the $u_{ij}$ are local symmetries (i.e. $u_{ij}=u_{ij}^*$ and $u_{ij}^2$ is a projection),
 \item the projections $u_{ij}^2$ fulfill $\sum_k u_{ik}^2=\sum_ku_{kj}^2=1$ for all $i,j$,
 \item in the case $\cC=\Katfat$, we also have $u_{ij}^2u_{kl}^2=u_{kl}^2u_{ij}^2$ for all $i,j,k,l$,
 \item in the case $\cC=\Katprim$, we even have $u_{ij}^2u_{kl}=u_{kl}u_{ij}^2$ for all $i,j,k,l$.
\end{enumerate}
\end{proposition}
\begin{proof} 
The \Cstar-algebras $\AKat$ fulfill the relations of $C(H_n^+)$, the \Cstar-algebra associated with the free hyperoctahedral quantum group $H_n^+$.  It follows that $u_{ij}^2$ is a projection, since $u_{ij}^2=u_{ij}^2(\sum_k u_{ik}^2)=\sum_ku_{ij}(u_{ij}u_{ik})u_{ik}=u_{ij}^4$.  Lemma \ref{lem:relations-and-intertwiners} shows that in addition the relations (iii) or (iv) hold, respectively.

In order to show, that $\AKat$ is universal with the above relations, note that $\rC(H_n^+)$ is universal with the relations in (i) and (ii).  So $\AKat$ is the quotient of this universal \Cstar-algebra, by the relations imposed by Lemma \ref{lem:relations-and-intertwiners}.  So it is the universal \Cstar-algebra for the relations (i), (ii) and (iii) or (i), (ii) and (iv), respectively.
\end{proof}

This proposition shows that the elements $u_{ij}^2$ fulfill the relations of $S_n$, the (classical) permutation group (or rather of $C(S_n)$). The squares of the elements $u_{ij}$ of the above \Cstar-algebras $\AKat$ thus behave like commutative elements, whereas the $u_{ij}$ itself behave like free elements. The quantum groups $\AKat$ are hence somewhat in between the commutative and the (purely) non-commutative world.

\begin{remark}
  From the description of their \Cstar-algebras in Proposition \ref{GestaltAprimary}, we can deduce that $\langle \fatcrosspart \rangle \neq \langle \primarypart \rangle$ by showing that the canonical quotient map from $\AKat$ to $\AKatprime$ for $\mathcal C=\Katfat$, $\mathcal C'=\Katprim$ is not an isomorphism.  Indeed take $n = 3$ and let $H = \CC^3$ with its canonical basis $e_1, e_2, e_3$.  Denote by $p_i$ the projection of $H$ onto $\CC e_i$, $i \in \{1, 2, 3\}$.  We define operators $w_{ij}$, $1 \leq i,j \leq 3$ on $H$, which define a representation of $\AKat$ but not of $\AKatprime$.  Let $w_{11}e_1 = e_2$, $w_{11}e_2 = e_1$ and $w_{11}e_3 = 0$ and $w_{22} = p_1$.  Then $w_{11}^2 w_{22}^2 = w_{22}^2 w_{11}^2$, but $w_{11} w_{22}^2 \neq w_{22}^2 w_{11}$.  Define
\begin{align*}
  w_{21} & = w_{12} = p_3 \eqcomma \quad w_{31} = w_{13} = 0 \\
  w_{23} & = w_{32} = p_2 \eqcomma \quad w_{33}  = p_1 + p_3 \eqstop
\end{align*}
Then $(w_{ij})_{ij}$ induces a representation of $\AKat$, which does not factor through $\AKat \ra \AKatprime$.
\end{remark}

\begin{remark}
\begin{enumerate}
\item If $\pi: \rC(H_n^{[\infty]}) \to \bo(H)$ is an irreducible representation, the projections $\pi(u_{ij}^2)$ are either 1 or 0, since they commute with all elements of $\pi(\rC(H_n^{[\infty]}))$.  Since $\sum_k\pi(u_{ik}^2)=\sum_k\pi(u_{ki}^2)=1$, there is a permutation $\gamma\in S_n$ such that $\pi(u_{ij}^2)=1$, if $\gamma(i)=j$, and $\pi(u_{ij}^2) = 0$, otherwise.  Recall, that the full group \Cstar-algebra $\Cstarmax(\Z_2^{*n})$ is isomorphic to the $n$-fold unital free product of the \Cstar-algebra $\Cstarmax(\Z_2)=\C^2$.  The latter is the universal \Cstar-algebra generated by a symmetry.  Thus, $\Cstarmax(\Z_2^{*n})$ is the universal unital \Cstar-algebra generated by $n$ symmetries $w_1,\ldots,w_n$.  So, the map $\pi_0(w_i):=\pi(u_{i\gamma(i)})$ defines a representation of $\Cstarmax(\Z_2^{*n})$.
\item Vice versa, we can produce representations of $\rC(H_n^{[\infty]})$ by permutations.  First note that the relations of $\rC(H_n^{[\infty]})$ are invariant under permutation of rows or columns of its fundamental corepresentation.  Let $\pi_0:\Cstarmax(\Z_2^{*n}) \to \bo(H)$ be a representation and let $\gamma\in S_n$ be a permutation.  The map $\rC(H_n^{[\infty]}) \ra \Cstarmax(\ZZ_2^{*n}): u_{ij} \mapsto \delta_{j \gamma(i)} \cdot w_i$ is well defined.  So $\pi(u_{ij}):=\delta_{j\gamma(i)}\pi_0(w_i)$ defines a (not necessarily irreducible) representation of $\rC(H_n^{[\infty]})$.
\end{enumerate}
\end{remark}

In the \Cstar-algebra associated with a simplifiable hyperoctahedral category $\cC$, the relations on the generators may be read directly from the partitions in single leg form. Let $p=a_{i(1)} \dotsm a_{i(k)} \in P(0,k)$ be a partition without upper points in single leg form. We consider $p$ as a word in the letters $a_1, \dotsc ,a_m$ (labelled from left to right). If we replace the letters $a_i$, $1 \leq i \leq m$, in $p$ by some choice of generators $u_{ij}$, $1 \leq i,j \leq n$, we obtain an element $a_{i(1)} \dotsm \, a_{i(k)} \in \AKat$; replacing the letters by the according elements $u_{ij}^2$ yields a projection $q \in \AKat$.

\begin{proposition}
\label{PropWordPartialIsom}
 Let $\mathcal C$ be a simplifiable hyperoctahedral category and let $p = a_{i(1)} \dotsm a_{i(k)}$ be a partition in single leg form.
 The following assertions are equivalent:
\begin{enumerate}
 \item $p \in \mathcal C$.
 \item $a_{i(1)} \dotsm \, a_{i(k)} = q$ in $\AKat$ for all choices $a_r \in \{u_{ij} \; | \; i,j = 1, \dotsc, n\}$, $1 \leq r \leq m$, where $q$ is the according range projection.
 \item For some $1 \leq s \leq k$, we have $q a_{i(1)}\dotsm \, a_{i(s)} = q a_{i(k)} \dotsm a_{i(s+1)}$ in $\AKat$ for all choices $a_r \in \{u_{ij} \; | \; i,j = 1, \dotsc, n\}$, $1 \leq r \leq m$, where $q$ is the according range projection.
\end{enumerate}
\end{proposition}
\begin{proof} 
The linear map $T_p: \C \to (\C^n)^{\otimes k}$ associated with $p$ is given by
\[
  T_p(1)
  =
  \sum_{i(1), \dotsc, i(k) = 1}^n
  {
    \delta_p(i)e_{i(1)} \otimes \dotsm \otimes e_{i(k)}
  }
  \eqstop
\]
We have
\[
  u^{\otimes k} (T_p \otimes 1)(1 \ot 1)
  =
  \sum_{i(1), \dotsc, i(k) = 1}^n
  {
    e_{i(1)} \otimes \dotsm \otimes e_{i(k)} \otimes
    \left( \sum_{j(1), \dotsc, j(k) = 1}^n \delta_p(j) \cdot u_{i(1) j(1)} \dotsm \, u_{i(k)j(k)} \right)
  }
  \eqcomma
\]
so that $p \in \mathcal C$, if and only if for all multi-indices $i=(i(1), \dotsc ,i(k))$ the equation:
\[
  \sum_{j(1), \dotsc, j(k) = 1}^n
  {
    \delta_p(j) \cdot u_{i(1) j(1)} \dotsm \, u_{i(k)j(k)}
  }
  =
  \delta_p(i)
\]
holds.  Now, assume (i) and let us show (ii).  Make a choice of $a_r \in \{ u_{ij} \amid i,j = 1, \dotsc, n \}$ for all ${r \in \{1, \dotsc, m\}}$.  Then there are multi-indices $i$ and $j$ satisfying $\delta_p(i) = \delta_p(j) = 1$ such that $a_{i(1)} \dotsm \, a_{i(k)} = u_{i(1)j(1)} \dotsm \, u_{i(k) j(k)}$.  Let $q$ be the projection given by $q := u_{i(1)j(1)}^2 \dotsm \, u_{i(k) j(k)}^2$.  Then
\begin{align*}
  u_{i(1)j(1)} \dotsm u_{i(k)j(k)}
  & = 
  u_{i(1)j(1)} \dotsm u_{i(k)j(k)} \left( \sum_{r(1), \dotsc, r(k) = 1}^n \delta_p(r) \cdot u_{i(k) r(k)} \dotsm \, u_{i(1)r(1)} \right) \\
  & =
  \sum_{r(1), \dotsc, r(k) = 1}^n
    \delta_p(r) \cdot u_{i(1)j(1)} \dotsm \, u_{i(k -1)j(k - 1)} \\
    & \qquad \qquad
    (\delta_{j(k) r(k)} u_{i(k)j(k)}^2)
    u_{i(k -1)r(k - 1)} \dotsm u_{i(1)r(1)}
   \\
  & = 
  \sum_{r(1), \dotsc, r(k) = 1}^n \delta_p(r) \cdot \delta_{i(1) r(1)} u_{i(1)j(1)}^2 \dotsm \, \delta_{i(k)r(k)} u_{i(k) j(k)}^2 \\
  & =
  q
  \eqstop
\end{align*}
This proves (ii). 
Conversely, assume (ii) and let $i$ be any multi-index.  If $\delta_p(i) = 0$, then $u_{i(1)j(1)} \dotsm \, u_{i(k)j(k)} = 0$ for any multi-index $j$ that satisfies $\delta_p(j) = 1$, since in this product there are at least two local symmetries that have mutually orthogonal support in the centre of $\AKat$.  Hence 
\[
  \sum_{j(1), \dotsc, j(k) = 1}^n \delta_p(j) \cdot u_{i(1)j(1)} \dotsm \, u_{i(k)j(k)}
  =
  0
  \eqstop
\]
Similarly, using the assumption (ii), if $\delta_p(i) = 1$, then 
\[
  u_{i(1)j(1)}^2 \dotsm \, u_{i(k)j(k)}^2
  =
  \begin{cases}
    u_{i(1)j(1)} \dotsm \, u_{i(k)j(k)} \eqcomma & \text{if } \delta_p(j) = 1 \eqcomma \\
    0 \eqcomma & \text{otherwise.}
  \end{cases}
\]
We obtain that
\[
  \sum_{j(1), \dotsc, j(k) = 1}^n \delta_p(j) \cdot u_{i(1)j(1)} \dotsm \, u_{i(k)j(k)}
  =
  \sum_{j(1), \dotsc, j(k) = 1}^n u_{i(1)j(1)}^2 \dotsm \, u_{i(k)j(k)}^2
  =
  1
  \eqstop
\]
This proves (i).  The assertions (ii) and (iii) are equivalent, since all projections $a_r^2$, $1 \leq r \leq m$ are absorbed by $q$ and $qa_{i(1)} \dotsm \, a_{i(k)} = a_{i(1)} \dotsm \, a_{i(k)}$.
\end{proof}

Let us recall the notion of a coopposite quantum group.  If $(A, \Delta)$ is a compact quantum group, then its coopposite version is the quantum group $(A, \Sigma \circ \Delta)$, where ${\Sigma: A \ot A \ra A \ot A}$ is the flip.  In particular, if $(A, u)$ is a compact matrix quantum group, its coopposite version is the compact matrix quantum group $(A, u^t)$.

\begin{corollary}
\label{cor:coopposite}
  Every easy quantum subgroup of $(\rC(H_n^{[\infty]}), u_{\mathrm{simpl}})$ is isomorphic to its coopposite version.
\end{corollary}
\begin{proof}
Let $(\AKat, u)$ be an easy quantum subgroup of $\rC(H_n^{[\infty]}, u_{\mathrm{simpl}})$ with corresponding category of partitions $\cC$.  By Proposition \ref{PropWordPartialIsom}, $(\AKat, u)$ is the universal \Cstar-algebra such that
\begin{itemize}
\item all $u_{ij}$ are local symmetries whose support is central in $\AKat$ and sums up to $1$ in every row and every column
\item for any partition $p = a_{i(1)} \dotsm \, a_{i(k)} \in \cC$ and any choice of elements $a_r \in \{u_{ij} \amid 1 \leq i,j \leq n \}$, $1 \leq r \leq m$ we have that $a_{i(1)} \dotsm \, a_{i(k)}$ is a projection.
\end{itemize}
These relations are invariant under taking the transpose of $u$.  So $u_{ij} \mapsto u_{ji}$ is a *-automorphism of $\AKat$.  This finishes the proof.
\end{proof}

\begin{example}
 Let $p$ be the word $p=abcbcacb$ and consider the \Cstar-algebras associated with the category $\langle \primarypart, p \rangle$.
Note that the equation $abcbcacb=1$ in $\Z_2 * \Z_2 * \Z_2$ is equivalent to $abcb=bcac$. The idea is that Proposition \ref{PropWordPartialIsom} yields the commutation relations
\[
  u_{ij} u_{kl} u_{rs} u_{kl} 
  =
  u_{kl} u_{rs} u_{ij} u_{rs}
\]
in $\AKat$ with $p\in \mathcal C$ ``wherever it makes sense''. To be more precise, we have \[qu_{ij}u_{kl}u_{rs}u_{kl}=qu_{kl}u_{rs}u_{ij}u_{rs},\]
where $q=u_{ij}^2u_{kl}^2u_{rs}^2$. By multiplying from left or right with the local symmetries $u_{ij}$, $u_{kl}$, and $u_{rs}$, we also obtain relations like for instance
\[
  q u_{ij} u_{kl} u_{rs} u_{kl} u_{rs} u_{ij}
  =
  q u_{kl} u_{rs}
  \quad \textnormal{or} \quad
  u_{ij} u_{kl} u_{rs} u_{kl} u_{rs} u_{ij} u_{rs} u_{kl}
  =
  q
  \eqstop
\]
\end{example}

\begin{remark}
The K-theory for the \Cstar-algebras associated with easy quantum groups is relatively unknown.  Voigt computed the K-theory for $O_n^+$ in \cite{voigt11-baum-connes} and there are some small extensions of this result to other easy quantum groups in \cite{weber12}.  Let us also mention \cite{vergniouxvoigt11-ktheory-free}.
\end{remark}


\section{Diagonal subgroups and their quantum isometry groups}
\label{sec:realation-qiso-groups}

If $\cC$ is a simplifiable hyperoctahedral category of partitions, the group $F(\cC)$ from Theorem \ref{thm:F-is-1-1} can be recovered directly from $(A_\cC(n), u_n)_{n \geq 2}$.  Vice versa, $A_\cC(n)$ arises as a natural subgroup of the quantum isometry group of $F(\cC)$.  This is explained by the following results. 

\begin{definition}
  For $H \subset \ZZ_2^{*\infty}$ write $(H)_n = \{ w \in H \amid w \text{ only involves the letters } a_1, \dotsc, a_n\}$.
\end{definition}

\begin{definition}
  Given a compact matrix quantum group $(A,u)$, its diagonal subgroup $\diag(A,u)$ is the discrete group that is generated by the image of the diagonal entries of $u$ in the quotient \Cstar-algebra $A/\langle u_{ij} \amid i \neq j\rangle$.
\end{definition}

\begin{theorem}
\label{thm:F-and-diagonal-subgroup}
Let $\cC$ be a simplifiable hyperoctahedral category of partitions.  Then $\diag(A_\cC(n), u_n) \cong \ZZ_2^{*n} / F(\cC)_n$
\end{theorem}
\begin{proof}
Let $H$ be the diagonal subgroup of $A_\cC(n)$.  The \Cstar-algebra $A_\cC(n)$ is the quotient of $A_o(n)$ by the relations $T_p \ot 1 = (u_n^o)^{\ot k}(T_p \ot 1)$ for all $p \in \cC$.  Moreover, the diagonal subgroup of the free orthogonal quantum group satisfies $\diag(A_o(n), u_n^o) = \ZZ_2^{*n}$.  Denote by $\pi$ the quotient homomorphism $A_\cC(n) \ra A_\cC(n)/\langle u_{ij} \amid i \neq j \rangle$. Then there is a commuting diagram
\[
\begin{xy}
  \xymatrix{
    (A_o(n), u_n^o) \ar[d] \ar[r] & \Cstarmax(\ZZ_2^{*n}) \ar[d]\\
    (A_{\cC}(n), u_n) \ar[r]^\pi     & \Cstarmax(H) \eqstop
  }
\end{xy}
\]
So $H$ is the universal group generated by elements $a_i = \pi(u_{ii})$, $1 \leq i \leq n$ of order two that satisfy $T_p \ot 1 = (\pi(u_n))^{\ot k}(T_p \ot 1)$ for all $p \in \cC$.  Moreover, it suffices to consider partitions $p \in \cC$ on one row.  Let $p \in \cC$ be such a partition of length $k$.  Then $T_p \ot 1 = u^{\ot k}(T_p \ot 1)$.  Hence, we have equality of
\[
  (T_p \ot 1)(1 \ot 1)
  =
  \sum_{i(1), \dotsc, i(k) = 1}^n
  {
    \delta_p(i) \cdot e_{i(1)} \ot \dotsm \ot e_{i(k)} \ot 1
  }
\]
and
\[
  u^{\ot k}(T_p \ot 1)(1 \ot 1)
  =
  \sum_{\substack{i(1), \dotsc, i(k), \\  j(1), \dotsc, j(k) = 1}}^n
  {
    \delta_p(j) \cdot e_{i(1)} \ot \dotsm \ot e_{i(k)} \ot u_{i(1)j(1)} \dotsm \, u_{i(k)j(k)}
  }
  \eqstop
\]
Applying the projection $\pi$ to both equations, we obtain that for all $i(1), \dotsc, i(k)$ with $\delta_p(i) = 1$ we have $\pi(u_{i(1)i(1)}) \dotsm \, \pi(u_{i(k)i(k)}) = 1$.  This shows that $H \cong \ZZ^{*n}/F(\cC)_n$.
\end{proof}

Note that as a consequence of the last theorem the \Cstar-algebra $A_\cC(n)$ is a canonical extension of the full group \Cstar-algebra $\Cstarmax(\ZZ_2^{* n}/F(\cC)_n)$.  

\begin{lemma}
\label{lem:inductive-limits}
  Let $H$ be a subgroup of $\ZZ_2^{*\infty}$.  Then $\ZZ_2^{*\infty}/H \cong \varinjlim (\ZZ_2^{*n}/(H)_n, \phi_n)$, where $\phi_n$ is defined by the diagram
\begin{center}
\setlength{\unitlength}{0.5cm}
\begin{picture}(10,3)
  \put(0,0){$\ZZ_2^{*n}$}
  \put(5,0){$\stackrel{a_i \mapsto a_i}{\lra}$}
  \put(10,0){$\ZZ_2^{*n+1}$}

  \put(0.5,1.8){\begin{rotate}{270} $\subset$ \end{rotate}}
  \put(10.5,1.8){\begin{rotate}{270} $\subset$ \end{rotate}}

  \put(0,2.5){$(H)_n$}
  \put(5,2.5){$\stackrel{\phi_n}{\lra}$}
  \put(10,2.5){$(H)_{n + 1}$}
 
  \put(12,0){.}
\end{picture}
\end{center}
\end{lemma}
\begin{proof}
  By universality of inductive limits, we have to show that for any compatible family of morphisms 
   \[
   \xymatrix{
     \ZZ_2^{*n}/(H)_n \ar[r]\ar[rrrd] &  \ZZ_2^{*n + 1}/(H)_{n+1} \ar[r]\ar[rrd] & \ZZ_2^{*n + 2}/(H)_{n+2} \ar[r]\ar[rd] & \dotsm \\
     &&& K
   }
   \]
the induced map $\pi: \ZZ_2^{*\infty} \ra K$ contains $H$ in its kernel.  This follows from the fact that $H = \bigcup_{n \geq 1} (H)_n$.
\end{proof}

The following corollary says that for a simplifiable hyperoctahedral category of partitions, we can recover $F(\cC)$ directly from the diagonal subgroups of the family $(A_{\cC}(n),u_n)_{n \geq 2}$.

\begin{corollary}
  Let $\cC$ be a simplifiable hyperoctahedral category of partitions.  Then $F(\cC) = \ker (\ZZ_2^{*\infty} \ra \varinjlim (\diag(A_{\cC}(n), u_n)))$.
\end{corollary}
\begin{proof}
  This follows from Theorem \ref{thm:F-and-diagonal-subgroup} and Lemma \ref{lem:inductive-limits}.
\end{proof}

We can also recover $(A_\cC(n), u_n)$ directly from the group $F(\cC)$ without passing through the framework of partitions.  This is done by considering quantum isometry groups.  By \cite{banicacurranspeicher09_2} the category $\Katprim$ gives rise to an easy quantum group, denoted by $H_n^{[\infty]}$.  It is the maximal simplifiable hyperoctahedral quantum group.

\begin{theorem}
\label{thm:qiso-groups}
  Let $H \leq E \leq \ZZ_2^{*\infty}$ be a proper $S_0$-invariant subgroup of $E$.  Then the maximal quantum subgroup of $(\cC(H_n^{[\infty]}), u_{\mathrm{simpl}})$ acting faithfully and isometrically on $\Cstarmax(\ZZ^{*n}_2/(H)_n)$ is $(A_{\cC_H}(n), u_n)$.
\end{theorem}
\begin{proof}
  Denote by $\pi: \ZZ_2^{*n} \ra \ZZ^{*n}_2/(H)_n$ the canonical quotient map and by $a_1, \dotsc, a_n$ the canonical generators of $\ZZ_2^{*n}$.  Since $H_n^{[\infty]}$ is isomorphic to its coopposite quantum group by Corollary \ref{cor:coopposite}, we may consider quantum subgroups of $(H_n^{[\infty]})^{\mathrm{coop}}$ instead of $H_n^{[\infty]}$ in the following.  Let $(A, u)$ be a compact quantum subgroup of $H_n^{[\infty]}$ such that $(A, u^t)$ acts faithfully by
\[
  \alpha:
  \Cstarmax(\ZZ_2^{*n}/(H)_n) \ra \Cstarmax(\ZZ_2^{*n}/(H)_n) \ot A:
  \pi(a_i) \ra \sum_j \pi(a_j) \ot u_{ij}
\]
on $\Cstarmax(\ZZ_2^{*n}/(H)_n)$ and preserves the length function $l$ associated with the generators $\pi(a_1), \dotsc, \pi(a_n)$.  We show that $(A,u)$ is a quantum subgroup of $(A_{\cC_H}(n), u_n)$.

Since $(A,u)$ is a quotient of $(C(H_n^{[\infty]}), u_{\mathrm{simpl}})$, Proposition \ref{GestaltAprimary} shows that the entries $u_{ij}$ of $u$ are self-adjoint partial isometries, which are pairwise orthogonal in every row and in every column and whose support projections are central.  We check the additional relations that are imposed on the $u_{ij}$ by the fact that $(A, u^t)$ acts isometrically on $\Cstarmax(\ZZ_2^{*n}/(H)_n)$.  For every word $a_{i(1)} \dotsm \, a_{i(k)}  \in \ZZ_2^{*n}$ we have
\begin{align*}
  \alpha(\pi(a_{i(1)} \dotsm \, a_{i(k)}))
  & =
  \sum_{j(1), \dotsc, j(k) = 1}^n
  {
    \pi(a_{j(1)} \dotsm \, a_{j(k)}) \ot u_{i(1)j(1)} \dotsm \, u_{i(k)j(k)}
  } 
  \eqstop
\end{align*}
This expression has non-zero coefficients $u_{i(1)j(1)} \dotsm \, u_{i(k)j(k)}$ only for those $(j(1), \dotsc, j(k))$ with $l(\pi(a_{j(1)} \dotsm \, a_{j(k)})) = l(\pi(a_{i(1)} \dotsm \, a_{i(k)}))$.  Since $l(\pi(a_{i(1)} \dotsm \, a_{i(k)})) = 0$ if and only if ${a_{i(1)} \dotsm \, a_{i(k)} \in (H)_n}$, this means in particular that for all $a_{i(1)} \dotsm \, a_{i(k)} \notin (H)_n$  we have
\[
  \sum_{a_{j(1)} \dotsm \, a_{j(k)} \in (H)_n}
  {
    u_{i(1)j(1)} \dotsm \, u_{i(k)j(k)}
  }
  =
  0
  \eqstop
\]
On the other hand, if $a_{i(1)} \dotsm \, a_{i(k)} \in (H)_n$ then
\[
  \alpha(\pi(a_{i(1)} \dotsm \, a_{i(k)}))
  =
  \alpha(1)
  =
  1 \ot 1
  \eqstop
\]
It follows that
\[
  \sum_{a_{j(1)} \dotsm \, a_{j(k)} \in (H)_n}
  {
    u_{i(1)j(1)} \dotsm \, u_{i(k)j(k)}
  }
  =
  1
  \eqstop
\]
We proved that for all $a_{i(1)} \dotsm \, a_{i(k)} \in \ZZ_2^{*n}$ we have
\[
\sum_{a_{j(1)} \dotsm \, a_{j(k)} \in (H)_n}
  {
    u_{i(1)j(1)} \dotsm \, u_{i(k)j(k)}
  }
  =
  \delta_p(i) \cdot 1
  \eqcomma
\]
where $p$ denotes the partition in $\cC_H$ that is associated with $a_{i(1)} \dotsm \, a_{i(k)}$.

Next, we claim that if $u_{i(1)j(1)} \dotsm \, u_{i(k)j(k)} \neq 0$, then the word $a_{j(1)} \dotsm \, a_{j(k)}$ arises from $a_{i(1)} \dotsm \, a_{i(k)}$ by a permutation of letters.
Recall that all $u_{ij}$ are self-adjoint partial isometries, which are pairwise orthogonal in every row and every column and whose supports are central in $A_{\cC_H}(n)$.  If $1 \leq \alpha, \beta \leq k$, we can conclude from $i(\alpha) = i(\beta)$ and $u_{i(1)j(1)} \dotsm \, u_{i(k)j(k)} \neq 0$, that $j(\alpha) = j(\beta)$.  Similarly, we conclude from $j(\alpha) = j(\beta)$, that $i(\alpha) = i(\beta)$.  This proves our claim.

Take now a partition $p \in \cC$.  By our claim, for any multi-index $i = (i(1), \dotsc, i(k))$, which satisfies $\delta_p(i) = 0$, we have
\[
  \sum_{j(1), \dotsc, j(k) =1}^n
  {
    \delta_p(j) \cdot u_{i(1)j(1)} \dotsm \, u_{i(k)j(k)}
  }
  =
  0
\]
If $i =  (i(1), \dotsc, i(k))$ satisfies $\delta_p(i) = 1$, we obtain, using the claim again,
\[
  \sum_{j(1), \dotsc, j(k) =1}^n
  {
    \delta_p(j) \cdot u_{i(1)j(1)} \dotsm \, u_{i(k)j(k)}
  }
  =
  \sum_{a_{j(1)} \dotsm \, a_{j(k)} \in (H)_n}
  {
    u_{i(1)j(1)} \dotsm \, u_{i(k)j(k)}
  }
  =
  1
  \eqstop
\]
Summarising, we have
\[
  \sum_{j(1), \dotsc, j(k) =1}^n
  {
    \delta_p(j) \cdot u_{i(1)j(1)} \dotsm \, u_{i(k)j(k)}
  }
  =
  \delta_p(i)
  \eqstop
\]
We infer that
\begin{align*}
  u^{\ot k}(T_p \ot 1)(1 \ot 1)
  & =
  \sum_{\substack{i(1), \dotsc, i(k), \\ j(1), \dotsc, j(k) = 1}}^n
  {
    e_{i(1)} \ot \dotsm \ot e_{i(k)} \ot \delta_p(j) \cdot u_{i(1)j(1)} \dotsm \, u_{i(k)j(k)}
  } \\
  & =
  \sum_{i(1), \dotsc, i(k) = 1}^n
  {
    e_{i(1)} \ot \dotsm \ot e_{i(k)} \ot
    \left( \sum_{j(1), \dots, j(k) = 1}^n
    {
      \delta_p(j) \cdot u_{i(1)j(1)} \dotsm \, u_{i(k)j(k)}
    }
    \right )    
  } \\
  & =
  \sum_{i(1), \dotsc, i(k) = 1}^n
  {
    e_{i(1)} \ot \dotsm \ot e_{i(k)} \ot \delta_p(i) \cdot 1
  } \\
  & =
  (T_p \ot 1)(1 \ot 1)
  \eqstop
\end{align*}
We proved that $(A, u)$ is a quantum subgroup of $(A_{\cC_H}(n), u_n)$.

Next we show that $(A_{\cC_H}(n), (u_n)^t)$ acts faithfully and isometrically on $\Cstarmax(\ZZ_2^{*n}/(H)_n)$.  The map
\[
  \alpha:
  \Cstarmax(\ZZ_2^{*n}/(H)_n)
  \ra
  \Cstarmax(\ZZ_2^{*n}/(H)_n) \ot A_{\cC_H}(n):
  \pi(a_i)
  \mapsto
  \sum_j \pi(a_j) \ot u_{ij}
\]
is a well defined action of $(A_{\cC_H}(n), (u_n)^t)$ on $\Cstarmax(\ZZ_2^n/(H)_n)$, by the calculations in the first part of the proof.  By definition, it is faithful, so it remains to show that it is isometric.  We say that a word $a_{i(1)} \dotsm \, a_{i(k)}$ is reduced in $\ZZ_2^{*n}/(H)_n$, if there is no $k' < k$ and $j(1), \dotsc, j(k')$ such that $\pi(a_{i(1)} \dotsm \, a_{i(k)}) = \pi(a_{j(1)} \dotsm \, a_{j(k')})$.  A word $a_{i(1)} \dotsm \, a_{i(k)}$ is reduced in $\ZZ_2^{*n}/(H)_n$ if and only if $l(\pi(a_{i(1)} \dotsm \, a_{i(k)})) = k$.  Denoting
\[
  L_k
  =
  \lspan \{
    w = \pi(a_{i(1)} \dotsm \, a_{i(k)})
    \amid
    a_{i(1)} \dotsm \, a_{i(k)} \text{ is reduced as a word in } \ZZ_2^{*n}/(H)_n
  \}
  \eqcomma
\]
we have to show that $\alpha(L_k) \subset L_k \ot A_{\cC_H}(n)$.  First note that $\alpha(L_k) \subset \bigcup_{k' \leq k} L_{k'} \ot A_{\cC_H}(n)$.  Let $\pi(a_{i(1)}) \dotsm \, \pi(a_{i(k)})$ denote a word that is reduced in $\ZZ_2^{*n}/(H)_n$, write $w = \pi(a_{i(1)} \dotsm \, a_{i(k)})$ and assume that $\alpha(w) \notin L_k \ot A_{\cC_H}(n)$.  We have
\[
  \alpha(w)
  = 
  \sum_{j(1), \dotsc, j(k)} \pi(a_{j(1)} \dotsm \, a_{j(k)}) \ot u_{i(1)j(1)} \dotsm \, u_{i(k)j(k)}
  \eqcomma
\]
so there is some mutiindex $(j(1), \dotsc, j(k))$ such that
\begin{itemize}
\item $a_{j(1)} \dotsm \, a_{j(k)}$ is not reduced in $\ZZ_2^{*n}/(H)_n$ and
\item $u_{i(1)j(1)} \dotsm \, u_{i(k)j(k)} \neq 0$.
\end{itemize}

As seen above, the word $a_{i(1)} \dotsm \, a_{i(k)}$ must arise as a permutation of letters from $a_{j(1)} \dotsm \, a_{j(k)}$, because $u_{i(1)j(1)} \dotsm \, u_{i(k)j(k)} \neq 0$.  Since $a_{j(1)} \dotsm \, a_{j(k)}$ is not reduced in $\ZZ_2^{*n}/(H)_n$ and $a_{i(1)} \dotsm \, a_{i(k)}$ arises from $a_{j(1)} \dotsm \, a_{j(k)}$ by a permutation of letters, also $a_{i(1)} \dotsm \, a_{i(k)}$ is not reduced in $\ZZ_2^{*n}/(H)_n$.  This is a contradiction.  We proved that $\alpha$ is an isometric action of $(A_{\cC_H}(n), (u_n)^t)$.  Summarising we showed that the maximal quantum subgroup of $(\rC(H_n^{[\infty]}), u_{\mathrm{simpl}})$ which acts faithfully and isometrically on $\Cstarmax(\ZZ_2^{*n}/(H)_n)$ is isomorphic to $(A_{\cC_H}(n), (u_n)^t)$.  Invoking Corollary \ref{cor:coopposite}, we see that $(A_{\cC_H}(n), (u_n)^t) \cong (A_{\cC_H}(n), u_n)$ and this finishes the proof.
\end{proof}

In view of the last theorem, it would be interesting to calculate the full quantum isometry groups of $\Cstarmax(\ZZ^{* n}/(H)_n)$.

\begin{example}
  One class of groups which appear as $\ZZ_2^{*n}/(H)_n$ for some $S_0$-invariant subgroup $H \leq E \leq \ZZ_2^{* \infty}$ are Coxeter groups.  A Coxeter group $G$ is of the above form if and only if
\[
  G =  \ZZ_2^{* n}/\langle (a_ia_j)^s \amid 1 \leq i,j \leq n \rangle \eqcomma
\]
for some $s \in \NN_{\geq 2}$.  The easy quantum group associated with it is $H_n^{[s]}$, since the category of partitions of the latter is given by $\langle \vierpart, h_s \rangle$. 
\end{example}

\subsection{The triangluar relationship between quantum groups, reflection groups and categories of partitions}
\label{sec:triangular-relationship}

Let us give a name to the groups that appear in the first part of this section.

\begin{definition}
  Let $H \leq \ZZ_2^{* \infty}$ be an $S_0$-invariant subgroup.  A \emph{symmetric reflection group} $G$ is the quotient of $\ZZ_2^{* n}$ by the intersection $H \cap \ZZ_2^{* n}$.  The images of the canonical generators of $\ZZ_2^{* n}$ in $G$ are called the generators of the symmetric reflection group $G$.
\end{definition}

We obtain a correspondence between simplifiable hyperoctahedral quantum groups and symmetric reflection groups.  By the work of Banica and Speicher \cite{banicaspeicher09},  there is a correspondence between easy quantum groups and categories of partitions.  Finally, the results of Section \ref{sec:group-theoretic-framework} show that there is a correspondence between symmetric reflection groups and simplifiable hyperoctahedral categories of partitions.  We therefore obtain a triangular correspondence between simplifiable hyperoctahedral  quantum groups, symmetric reflection groups and simplifiable hyperoctahedral categories of partitions that we are going to recap.

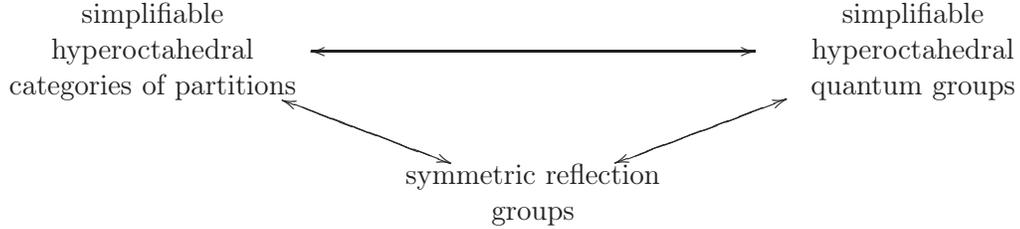
\hspace{1.5cm}
\begin{xy}
\xymatrix{
  \begin{minipage}{4cm}
    \begin{center}
      simplifiable hyperoctahedral categories of partitions
    \end{center}
  \end{minipage}
  \ar[dr]
  \ar[rr]
  & &
  \begin{minipage}{4cm}
    \begin{center}
      simplifiable hyperoctahedral quantum groups
    \end{center}
  \end{minipage}
  \ar[ll]
  \ar[dl]
  \\ &
  \begin{minipage}{4cm}
    \begin{center}
      symmetric reflection groups
    \end{center}
  \end{minipage}
  \ar[ur]
  \ar[ul]
}
\end{xy}

\textbf{From quantum groups to categories of partitions and back:} By the work of Banica and Speicher \cite{banicaspeicher09} there is a one-to-one correspondence between easy quantum groups and categories of partitions.  As described in Section \ref{sec:categories-of-partitions}, the category of partitions associated with an easy quantum group $(A, u)$ describes the intertwiner spaces between tensor powers of the fundamental corepresentation $u$.  By definition, an easy quantum groups is called a simplifiable hyperoctahedral quantum group, if the category of partitions associated with it contains the four block $\vierpart$, the pair positioner partition $\primarypart$, but not the double singleton $\singleton \ot \singleton$. 

\textbf{From categories of partitions to reflection groups and back:} Theorem \ref{thm:F-is-1-1} shows that there is a one-to-one correspondence between categories of partitions and $S_0$-invariant subgroups of $\ZZ_2^{*\infty}$.  By definition, symmetric reflection groups on infinitely many generators correspond precisely to the $S_0$-invariant subgroups of $\ZZ_2^{* \infty}$.

\textbf{From quantum groups to reflection groups and back:}  Let $(A, u)$, $u \in \Cmat{n}(A)$ be a compact matrix quantum group.  The quotient of $A$ by the ideal generated by $\{u_{ij} \amid i \neq j \}$ is a cocommutative compact matrix quantum group.  It is of the form $(\Cstar(G), (\delta_{ij} g_i)_{ij})$ for a generating set $g_1, \dotsc, g_n$ of a discrete group $G$.  Theorem \ref{thm:F-and-diagonal-subgroup} shows that if $(\rC(H), u)$ is a simplifiable hyperoctahedral quantum group, then the associated discrete group $G$ with generators $g_1, \dotsc, g_n$ is a symmetric reflection group.  Vice versa, Theorem \ref{thm:qiso-groups} associates with a symmetric reflection group $G$ on finitely many generators $g_1, \dotsc, g_n$ the maximal quantum subgroup of $H \subset H_n^{[\infty]}$ that acts faithfully and isometrically on $\Cstarmax(G)$.  The remark after Theorem \ref{thm:F-and-diagonal-subgroup} says that $\rC(H)$ is a canonical extension of $\Cstarmax(G)$ as a \Cstar-algebra.

\textbf{From finitely generated to infinitely generated symmetric reflection groups and back:}  Given a symmetric reflection group $G$ on infinitely many generators $g_1, g_2, \dotsc$ all groups $G_n = \langle g_1, g_2, \dotsc g_n \rangle$ are symmetric reflection groups also.  They form an inductive system
\[
  \dotsm \hra G_n \hra G_{n+1} \hra \dotsm
\]
of symmetric reflection groups.  We can obtain $G$ as the inductive limit of this system.

Theorems \ref{thm:F-and-diagonal-subgroup} and \ref{thm:qiso-groups} show that the above correspondences are compatible with each other.  Put differently, the triangle between quantum groups, discrete groups and categories of partitions commutes.




{\small \parbox[t]{200pt}{Sven Raum\\ KU Leuven, Department of Mathematics\\
   Celestijnenlaan 200B\\ B--3001 Leuven \\ Belgium
   \\ {\footnotesize sven.raum@wis.kuleuven.be}}
\hspace{15pt}
\parbox[t]{200pt}{Moritz Weber\\ Saarland University, Fachbereich Mathematik\\
   Postfach 151159\\ 66041 Saarbr{\"u}cken \\ Germany
   \\ {\footnotesize weber@math.uni-sb.de}}

\end{document}